\documentclass[11pt]{amsart}
\usepackage[margin=1in,letterpaper]{geometry}
\usepackage{enumitem,amsthm,amsmath,amssymb,mathrsfs,graphicx,esint,array,pgf,tikz,lastpage,hyperref,comment}
\usetikzlibrary{arrows,shapes}
\usetikzlibrary{arrows.meta,bending,automata}
\usetikzlibrary{decorations.pathreplacing}

\newtheoremstyle{mythm}
{.5\baselineskip}	
{.5\baselineskip}	
{}		
{}		
{\bf}	
{. }		
{ }		
{}		

\theoremstyle{mythm}
\newtheorem{theorem}{Theorem}[section] 
\newtheorem{lemma}[theorem]{Lemma}
\newtheorem{proposition}[theorem]{Proposition}
\newtheorem{corollary}[theorem]{Corollary}
\newtheorem{definition}[theorem]{Definition}
\newtheorem{example}[theorem]{Example}

\newtheorem{remark}[theorem]{Remark}
\newtheorem{question}{Question}

\setlist[itemize]{leftmargin=2em}

\newcommand{\notimplies}{\;\not\!\!\!\implies}

\title{An Excursion with Divergence Properties}
\author{Christopher Caruvana}
\address{School of Sciences, Indiana University Kokomo, 2300 S. Washington Street, Kokomo, IN 46902 USA}
\email{caruvana@gmail.com}
\urladdr{https://chcaru.pages.iu.edu/}
\date{\today}
\keywords{closed discrete selection, selectively highly divergent, statistical convergence of order \(\alpha\), selection principles, selection games, limited-information strategies}
\subjclass{54A20, 91A44, 40A35, 54B20, 54C35}

\begin{document}

\begin{abstract}
	In this note, we compare and contrast various selective divergence properties
	such as the properties of being discretely selective and selectively highly divergent.
	We identify and incorporate a class of subsemigroups of the semigroup of strictly increasing maps
	from the naturals to themselves.
	We investigate certain implications for hyperspaces of finite subsets and characterize the closed discrete selection
	game on a space in terms of a particular selection game on the Vietoris hyperspace of finite subsets of that space.
	We also isolate some sufficient conditions on a space that guarantee that the corresponding Pixley-Roy hyperspace
	of finite subsets is discretely selective.
	We end by noting that the properties of being discretely selective and of being selectively highly divergent are equivalent
	in rings of continuous functions with standard topologies of uniform convergence.
\end{abstract}

\maketitle

\section{Introduction}

A space \(X\) is said to be \emph{discretely selective} if, given any sequence \(\langle U_n : n \in \mathbb N \rangle\)
of nonempty open subsets of \(X\), there exists a selection \(x_n \in U_n\) for each \(n \in \mathbb N\) such that
\(\{ x_n : n \in \mathbb N \}\) is a closed and relatively discrete subset of \(X\).
The property of being discretely selective first appeared in \cite[Lemma 3.8.1]{GuerreroSanchezTkachuk} and later gained its name in \cite{Tkachuk};
it has been studied mostly in the context of continuous function spaces, as evidenced by the references above.

A space \(X\) is said to be \emph{selectively highly divergent}, or SHD in short, if, given any sequence \(\langle U_n : n \in \mathbb N \rangle\)
of nonempty open subsets of \(X\), there exists a selection \(x_n \in U_n\) for each \(n \in \mathbb N\) such that
the sequence \(\langle x_n : n \in \mathbb N \rangle\) has no convergent subsequence.
The property of being selectively highly divergent first appeared in \cite{SHDFirst} and was expanded on in both \cite{BellaSpadaro} and \cite{RemarksOnSHD}.

Note that, if the sequence of selections forms an injective function \(\mathbb N \to X\), then the former property implies the latter property.
This property is our inspiration for Definition \ref{def:InjDiscSel}.
However, any discrete space is discretely selective but not SHD.
This leads one to wonder exactly under what conditions we can guarantee an implication between the two properties.
Observe also that these two properties give rise to particular topological selection games, which will be the principal protagonists throughout this work.

As a slight modification to the notion of SHD, appearing in \cite{RemarksOnSHD},
a space \(X\) is said to be \emph{weakly selectively highly divergent}, or WSHD in short,
if, given any sequence \(\langle U_n : n \in \mathbb N \rangle\) of nonempty open subsets of \(X\), there exists an infinite set \(A \subseteq \mathbb N\)
and a selection \(x_n \in U_n\) for each \(n \in A\) such that the sequence \(\langle x_n : n \in A \rangle\) has no convergent subsequence.
In \cite{RemarksOnSHD}, Question 4 asks whether the property of being SHD is equivalent to WSHD.
We will reflect a bit on this question and expand it to other contexts.

The primary objective of this note is to compare and contrast the above-mentioned notions in some level of detail.
Part of this evaluation is done by providing \(\mathsf{ZFC}\)
examples of spaces that disagree on the properties of discrete selectivity and SHD, but also noting that the two properties are equivalent
in the context of continuous function rings with certain topologies of uniform convergence (Theorem \ref{thm:MainFunctionTheorem}).
Despite their equivalence in continuous function rings, we give examples establishing that these properties need not be equivalent in the more
general class of all topological groups (Examples \ref{ex:CirclePower} and \ref{ex:2splitting}).

Another aspect of this evaluation is to investigate situations in which the Pixley-Roy hyperspace \(\mathcal F[X]\) of finite subsets of a space \(X\)
enjoys discrete selectivity.
This topic in the context of SHD has been thoroughly developed through \cite{SHDFirst,BellaSpadaro,RemarksOnSHD}, and we extend this genre of research
by isolating particular sufficient conditions on \(X\) to guarantee that \(\mathcal F[X]\) is discretely selective (Theorem \ref{thm:SufficientFXDS}).

We also elaborate on the theory of the hyperspace \(\mathcal P_{\mathrm{fin}}(X)\) of finite subsets of \(X\) viewed as a subspace of the Vietoris
hyperspace \(\mathbb K(X)\) of compact subsets.
This hyperspace has proved to be useful in capturing certain properties typically couched in terms of a property being enjoyed by all finite powers of
a space (see \cite{CHVietoris,C_HurewiczVietoris}).
In this context, we offer an equivalence between the closed discrete selection game on \(X\) and a particular selection game on \(\mathcal P_{\mathrm{fin}}(X)\)
(Theorem \ref{thm:PfinCD}).
As a particular consequence of this equivalence is the fact that, if \(X\) is discretely selective, then so is \(\mathcal P_{\mathrm{fin}}(X)\).
We also show, through a general strategic transferral result (Theorem \ref{thm:HyperSDS}), that, if \(X\) is SHD or WSHD, then so is \(\mathcal P_{\mathrm{fin}}(X)\).

In our efforts to reflect on potential differences between the SHD and WSHD properties, we have included Section \ref{sec:Asymptotic}
primarily as a way to robustly stratify the defining conditions.
We have focused our attention in this section on establishing some basic structural results for a particular variant of asymptotic density,
\(\alpha\)-density introduced in \cite{Colak}.
The main contribution of this section is Theorem \ref{thm:Dalpha} which asserts that, given two subsets of the naturals expressed as corresponding
enumerations \(\phi\) and \(\psi\) which have complements that are of \(\alpha\)-density zero, then the subset of the naturals corresponding
to \(\phi \circ \psi\) also has a complement of \(\alpha\)-density zero.
This mirrors a classical fact about the traditional asymptotic density.
We incorporate these notions in results throughout this work.

In the final section, we expand the contents of \cite[Theorem 13]{RemarksOnSHD} to include the discretely selective property on \(C_p(X)\)
as a corollary to a more general result (Theorem \ref{thm:MainFunctionTheorem}).
From the more general result, we obtain the analogous equivalence of properties relative to \(C_k(X)\), as well.

Throughout, we include, as relevant, questions for further investigation.

In the interest of space, we have limited historical remarks and trust that the reader may find additional information
and context in the references of the papers cited within.

\section{Background and Preliminaries}

\subsection{Basics}

For convenience, we will use \(\mathbb N\) to represent the set of \emph{positive} integers.
This choice will be motivated by Remark \ref{rmk:BasicAsymptoticThing}.
We will use \(\omega\) to represent the first infinite ordinal, and we will identify \(\mathbb N \cup \{0\}\) with \(\omega\).
For sets \(A\) and \(B\), we use \(A^B\) to be the set of all functions \(B \to A\);
we will use \(A^{<\mathbb N}\) to denote \[\{ \langle \rangle \} \cup \bigcup \{ A^{\{j \in \mathbb N : j \leq n \}} : n \in \mathbb N \}.\]
We also use the notation \(\# A\) to refer to the cardinality of the set \(A\).
If \(\kappa\) is a cardinal number and \(X\) is a set, we will use \([X]^\kappa\) to denote the set of all \(A \subseteq X\) with \(\# A = \kappa\).
Similarly, \([X]^{<\kappa}\) will be used to denote the set of all \(A \subseteq X\) with \(\# A < \kappa\);
when relevant, we will, without explicit mention, omit \(\varnothing\) from \([X]^{<\kappa}\) to avoid certain trivialities.
Given a set \(X\), we will use the phrase \emph{faithful enumeration} to refer to a bijection with a cardinal number.
We refer the reader to \cite{Kunen} for particular notation regarding cardinal numbers and \cite{Engelking} for any topological notions
herein undefined.
Throughout, for a function \(s : \mathbb N \to X\) for a set \(X\), we will use \(s_n\) and \(s(n)\) interchangeably.

By a \emph{space}, we mean any nonempty topological space.
Any assumptions regarding separation axioms will be explicitly stated when relevant.
For a space \(X\) and a set \(A\), we will use the notation \(X^A\) to represent the set
of all functions \(A \to X\) endowed with the usual Tychonoff product topology.

For a space \(X\), we will use \(\mathscr T_X\) to represent the set of all nonempty open subsets of \(X\).
For a point \(x\) of the space \(X\), we denote the neighborhood filter at \(x\) by \(\mathcal N_{X,x} = \{ U \in \mathscr T_X : x \in U \}\).

For a sequence \(s \in X^{\mathbb N}\) of a space \(X\) and a point \(x \in X\), we write \(s \to x\) to mean that \(s\) converges to \(x\);
that is, \(s \to x\) if, for every \(U \in \mathcal N_{X,x}\), there exists \(M \in \mathbb N\) such that, for every \(n \geq M\),
\(s(n) \in U\).
Let \(\Gamma_{X,x} = \{ s \in X^{\mathbb N} : s \to x \}\) and \(\mathrm{CS}_X = \bigcup \{ \Gamma_{X,x} : x \in X \}\).
We use the notation \(\mathrm{CS}_X\) to reflect the phrase \emph{convergent sequences of \(X\)}.

As the final notational convention we introduce in this subsection, we let \(\mathrm{CD}_X\) denote the set of all closed and relatively discrete
subsets of \(X\); that is, \(\mathrm{CD}_X\) consists of all \(A \subseteq X\) such that \(A\) is closed and, for every \(x \in X\), there
exists \(U\in\mathscr T_X\) such that \(U \cap A = \{ x \}\).

Throughout, we will also use symbols like \(\mathscr T\) and \(\mathrm{CD}\) to function as \emph{topological operators};
that is, functions on the class of topological spaces that produce sets relevant to the given space's topological structure.

\subsection{Selection Principles and Games}
Selection principles and topological games have been topics of general interest for many years.
For a general overview of selection principles, see \cite{KocinacSelectedResults,ScheepersNoteMat,ScheepersSelectionPrinciples};
for more on the history of topological games, see \cite{TelgarskySurvey}.

\begin{definition}
	For topological operators or sets \(\mathcal A\) and \(\mathcal B\),
	\[
		\mathsf{S}_1(\mathcal A, \mathcal B)
		\equiv \left( \forall A \in \mathcal A^{\mathbb N} \right) \left( \exists x \in \prod_{n\in\mathbb N} A_n \right) \{x_n : n \in \mathbb N\} \in \mathcal B.
	\]
\end{definition}
As we will see here, it is sometimes important to retain the selections in their sequential form.
\begin{definition}
	For topological operators or sets \(\mathcal A\) and \(\mathcal B\),
	\[
		\vec{\mathsf{S}}_1(\mathcal A, \mathcal B)
		\equiv \left( \forall A \in \mathcal A^{\mathbb N} \right) \left( \exists x \in \prod_{n\in\mathbb N} A_n \right) \langle x_n : n \in \mathbb N \rangle \in \mathcal B.
	\]
\end{definition}
Following \cite{ScheepersNoteMat}, for a space \(X\) and topological operators \(\mathcal A\) and \(\mathcal B\),
we write \(X \models \mathsf S^\ast_1(\mathcal A, \mathcal B)\), where \(\mathsf S^\ast \in \left\{\mathsf S, \vec{\mathsf S} \right\}\),
to mean that \(X\) satisfies the selection principle \(\mathsf S^\ast_1(\mathcal A_X, \mathcal B_X)\).
\begin{remark}
	For a space \(X\) and a set \(\mathcal B \subseteq X^{\mathbb N}\), the property \(\vec{\mathsf S}_1(\mathscr T_X, \mathcal B)\)
	is equivalent to the assertion that \(\mathcal B\) is dense in \(\mathsf b(X^{\mathbb N})\), where \(\mathsf b(X^{\mathbb N})\) represents
	the set \(X^{\mathbb N}\) endowed with the box topology.
\end{remark}
\begin{remark}
	A space \(X\) is discretely selective if and only if \(X \models \mathsf S_1(\mathscr T, \mathrm{CD})\).
\end{remark}

Selection principles naturally give rise to topological selection games.
\begin{definition}
	For topological operators or sets \(\mathcal A\) and \(\mathcal B\), the single-selection game \(\mathsf G_1(\mathcal A, \mathcal B)\) is as follows.
	There are two players, P1 and P2, and the game has \(\mathbb N\) rounds.
	In the \(n^{\mathrm{th}}\) round, P1 chooses \(A_n \in \mathcal A\) and P2 responds with \(x_n \in A_n\).
	P2 is declared the winner if \(\{x_n : n \in \mathbb N\} \in \mathcal B\).
	Otherwise, P1 wins.
\end{definition}
We will reference the finite-selection game in this work as well.
\begin{definition}
	For topological operators or sets \(\mathcal A\) and \(\mathcal B\), the finite-selection game \(\mathsf G_{\mathrm{fin}}(\mathcal A, \mathcal B)\) is as follows.
	There are two players, P1 and P2, and the game has \(\mathbb N\) rounds.
	In the \(n^{\mathrm{th}}\) round, P1 chooses \(A_n \in \mathcal A\) and P2 responds with \(F_n \in [A_n]^{<\aleph_0}\).
	P2 is declared the winner if \(\bigcup\{F_n : n \in \mathbb N\} \in \mathcal B\).
	Otherwise, P1 wins.
\end{definition}
As with selection principles, we will require the sequential form of single-selection games throughout this work.
\begin{definition}
	For topological operators or sets \(\mathcal A\) and \(\mathcal B\), the game \(\vec{\mathsf G}_1(\mathcal A, \mathcal B)\) is as follows.
	There are two players, P1 and P2, and the game has \(\mathbb N\) rounds.
	In the \(n^{\mathrm{th}}\) round, P1 chooses \(A_n \in \mathcal A\) and P2 responds with \(x_n \in A_n\).
	P2 is declared the winner if \(\langle x_n : n \in \mathbb N\rangle \in \mathcal B\).
	Otherwise, P1 wins.
\end{definition}
For additional commentary on the sequential versions of selection principles and games, see \cite{CHStar}.

We will use the phrase \emph{selection games} to refer to the collection of all games defined above.
We will refer to \(\mathsf G_1(\mathscr T_X, \mathrm{CD}_X)\) as the \emph{closed discrete selection game on \(X\)}.

Any game generates notions of strategies.
We will be considering here full- and limited-information strategies.
Since the sequential single-selection game is the principal game in this work, we will define the relevant strategy types for it
explicitly here and refer the reader to \cite{CCH_Menger} for more on the other games.
\begin{definition}
    We define strategies of various strengths below.
    \begin{itemize}
    \item
    A \emph{strategy for P1} in \(\vec{\mathsf{G}}_1(\mathcal A, \mathcal B)\) is a function
    \(\sigma:(\bigcup \mathcal A)^{<\mathbb N} \to \mathcal A\).
    A strategy \(\sigma\) for P1 is called \emph{winning} if whenever \(x_n \in \sigma\langle x_k : 1 \leq k < n \rangle\)
    for all \(n \in \mathbb N\), \(\langle x_n: n\in\mathbb N\rangle \not\in \mathcal B\).
    If P1 has a winning strategy, we write \(\mathrm{I} \uparrow \vec{\mathsf{G}}_1(\mathcal A, \mathcal B)\).
    \item
    A \emph{strategy for P2} in \(\vec{\mathsf{G}}_1(\mathcal A, \mathcal B)\) is a function
    \(\tau:\mathcal A^{<\mathbb N} \to \bigcup \mathcal A\) such that \(\tau(A_1,\ldots,A_n) \in A_n\) for any \(n \in \mathbb N\)
    and any finite sequence \(\langle A_1 , \ldots, A_n \rangle\) of \(\mathcal A\).
    A strategy \(\tau\) for P2 is \emph{winning} if, whenever \(A_n \in \mathcal A\) for all \(n \in \mathbb N\),
    \(\langle \tau(A_1, \ldots, A_n) : n \in \mathbb N \rangle \in \mathcal B\).
    If P2 has a winning strategy, we write \(\mathrm{II} \uparrow \vec{\mathsf{G}}_1(\mathcal A, \mathcal B)\).
    \item
    A \emph{predetermined strategy} for P1 is a strategy which only considers the current turn number.
    Formally, it is a function \(\sigma: \mathbb N \to \mathcal A\).
    If P1 has a winning predetermined strategy, we write \(\mathrm{I} \underset{\mathrm{pre}}{\uparrow} \vec{\mathsf{G}}_1(\mathcal A, \mathcal B)\).
    \item
    A \emph{Markov strategy} for P2 is a strategy which only considers the most recent move of P1 and the current turn number.
    Formally, it is a function \(\tau :\mathcal A \times \mathbb N \to \bigcup \mathcal A\) such that \(\tau(A,n) \in A\) for every \(A \in \mathcal A\).
    If P2 has a winning Markov strategy, we write
    \(\mathrm{II} \underset{\mathrm{mark}}{\uparrow} \vec{\mathsf{G}}_1(\mathcal A, \mathcal B)\).
    \item
    If there is a single element \(A_0 \in \mathcal A\) so that the constant function with value \(A_0\) is a winning strategy for P1,
    we say that P1 has a \emph{constant winning strategy}, denoted by
    \(\mathrm{I} \underset{\mathrm{cnst}}{\uparrow} \vec{\mathsf{G}}_1(\mathcal A, \mathcal B)\).
    \end{itemize}
\end{definition}
Note that, for any selection game \(\mathcal G\),
\[
    \mathrm{II} \underset{\mathrm{mark}}{\uparrow} \mathcal G
    \implies \mathrm{II} \uparrow \mathcal G
    \implies \mathrm{I} \not\uparrow \mathcal G
    \implies \mathrm{I} \underset{\mathrm{pre}}{\not\uparrow} \mathcal G
    \implies \mathrm{I} \underset{\mathrm{cnst}}{\not\uparrow} \mathcal G.
\]
\begin{definition}
    For two selection games \(\mathcal G\) and \(\mathcal H\), we write \(\mathcal G \leq_{\mathrm{II}} \mathcal H\)
    if each of the following hold:
    \begin{itemize}
        \item 
        \(\mathrm{II} \underset{\mathrm{mark}}{\uparrow} \mathcal G \implies \mathrm{II} \underset{\mathrm{mark}}{\uparrow} \mathcal H\),
        \item 
        \(\mathrm{II} \uparrow \mathcal G \implies \mathrm{II} \uparrow \mathcal H\),
        \item 
        \(\mathrm{I} \not\uparrow \mathcal G \implies \mathrm{I} \not\uparrow \mathcal H\), and
        \item 
        \(\mathrm{I} \underset{\mathrm{pre}}{\not\uparrow} \mathcal G \implies \mathrm{II} \underset{\mathrm{pre}}{\not\uparrow} \mathcal H\).
    \end{itemize}
    If, in addition,
    \begin{itemize}
        \item 
        \(\mathrm{I} \underset{\mathrm{cnst}}{\not\uparrow} \mathcal G \implies \mathrm{II} \underset{\mathrm{cnst}}{\not\uparrow} \mathcal H,\)
    \end{itemize}
    we write that \(\mathcal G \leq^+_{\mathrm{II}} \mathcal H\).
\end{definition}
\begin{remark} \label{remark:LindelofAndSelection}
    The following are mentioned in \cite[Prop. 15]{ClontzDualSelection} and \cite[Lemma 2.12]{CHVietoris}:
    \begin{itemize}
        \item 
        \(\mathrm{I} \underset{\mathrm{pre}}{\not\uparrow} \mathsf{G}_1(\mathcal{A},\mathcal{B})\)
        is equivalent to \(\mathsf{S}_1(\mathcal A, \mathcal B)\).
        \item
        \(\mathrm{I} \underset{\mathrm{cnst}}{\not\uparrow} \mathsf{G}_1(\mathcal{A},\mathcal{B})\)
        is equivalent to the property that, for every \(A \in \mathcal A\), there is \(B \in [A]^{\leq \aleph_0}\)
        so that \(B \in \mathcal B\).
    \end{itemize}
    In a similar way, one can see that \(\mathrm{I} \underset{\mathrm{pre}}{\not\uparrow} \vec{\mathsf{G}}_1(\mathcal{A},\mathcal{B})\)
    is equivalent to \(\vec{\mathsf{S}}_1(\mathcal A, \mathcal B)\).
\end{remark}
In particular, note that, if \(\vec{\mathsf G}_1(\mathcal A_1, \mathcal B_1) \leq_{\mathrm{II}} \vec{\mathsf G}_1(\mathcal A_2, \mathcal B_2)\), then,
for any space \(X\), \(X \models \vec{\mathsf S}_1(\mathcal A_1, \mathcal B_2)\) implies that \(X \models \vec{\mathsf S}_1(\mathcal A_2, \mathcal B_2)\).

Observe that \(\leq_{\mathrm{II}}\) and \(\leq_{\mathrm{II}}^+\) are both transitive and reflexive.
Hence, they induce equivalence relations on the class of selection games.
\begin{definition}
	For selection games \(\mathcal G\) and \(\mathcal H\), we will write
	\begin{itemize}
		\item
		\(\mathcal G \equiv \mathcal H\) if \(\mathcal G \leq_{\mathrm{II}} \mathcal H\) and \(\mathcal H \leq_{\mathrm{II}} \mathcal G\);
		\item
		\(\mathcal G \rightleftarrows \mathcal H\) if \(\mathcal G \leq_{\mathrm{II}}^+ \mathcal H\) and \(\mathcal H \leq_{\mathrm{II}}^+ \mathcal G\).
	\end{itemize}
\end{definition}

We note that the sequential single-selection version of \cite[Corollary 2.17]{CHVietoris} obtains.
\begin{theorem} \label{thm:TranslationTheorem}
	Let \(\mathcal A\), \(\mathcal B\), \(\mathcal C\), and \(\mathcal D\) be collections.
	Suppose, for each \(n \in \mathbb N\), there are functions
	\begin{itemize}
		\item
		\(\overleftarrow{\mathsf T}_{\mathrm{I},n} : \mathcal B \to \mathcal A\) and
		\item
		\(\overrightarrow{\mathsf T}_{\mathrm{II},n} : \left( \bigcup \mathcal A \right) \times \mathcal B \to \bigcup \mathcal B\)
	\end{itemize}
	such that the following properties hold:
	\begin{itemize}
		\item
		for any \(B \in \mathcal B\) and \(x \in \overleftarrow{\mathsf T}_{\mathrm{I},n}(B)\), \(\overrightarrow{\mathsf T}_{\mathrm{II},n}(x,B) \in B\) and
		\item
		if \(x_n \in \overleftarrow{\mathsf T}_{\mathrm{I},n}(B_n)\) for each \(n \in \mathbb N\) and
		\(\langle x_n : n \in \mathbb N \rangle \in \mathcal C\), then
		\(\left\langle \overrightarrow{\mathsf T}_{\mathrm{II},n}(x_n,B_n) : n \in \mathbb N \right\rangle \in \mathcal D\).
	\end{itemize}
	Then \(\vec{\mathsf G}_1(\mathcal A, \mathcal C) \leq_{\mathrm{II}} \vec{\mathsf G}_1(\mathcal B, \mathcal D)\).
	
	If, in addition, \(\overleftarrow{\mathsf T}_{\mathrm{I},n} = \overleftarrow{\mathsf T}_{\mathrm{I},m}\) for all \(n,m\in\mathbb N\),
	then \(\vec{\mathsf G}_1(\mathcal A, \mathcal C) \leq^+_{\mathrm{II}} \vec{\mathsf G}_1(\mathcal B, \mathcal D)\).
\end{theorem}
The proof, which we leave to the reader, is nearly identical to the proof of \cite[Theorem 12]{CHContinuousFunctions}, where the constant strategy portion is similar to
what is provided in \cite[Theorem 2.16]{CHVietoris}.

\subsection{Hyperspaces of Finite Subsets}

We will be considering two standard ways to topologize the set \([X]^{<\aleph_0}\) of nonempty finite subsets of a given space \(X\).
Though the square brackets \([ \ ]\) are overloaded in various ways, context of use should dispel any potential confusion.

Pixley-Roy hyperspaces have received detailed attention in the literature.
For a general development of the theory of Pixley-Roy hyperspaces, see \cite{DouwenPR}.
\begin{definition}
	Let \(\mathcal F[X]\) represent \([X]^{<\aleph_0}\) endowed with the Pixley-Roy topology;
	that is, the topology which has as a basis sets of the form
	\[[F,U] = \{ G \in [X]^{<\aleph_0} : F \subseteq G \subseteq U \},\]
	where \(F \subseteq [X]^{<\aleph_0}\), and \(U \in \mathscr T_X\),
\end{definition}

For a space \(X\), the hyperspace \(\mathbb K(X)\) of nonempty compact subsets of \(X\) endowed with the Vietoris topology
has a rich history.
For a thorough treatment of the Vietoris topology, see \cite{MichaelSubsets}.
In this work, we will focus only on the subspace of \(\mathbb K(X)\) consisting of finite subsets of \(X\).
\begin{definition}
	Let \(\mathcal P_{\mathrm{fin}}(X)\) represent \([X]^{<\aleph_0}\) viewed as a subspace of \(\mathbb K(X)\).
	Note that a basis for \(\mathcal P_{\mathrm{fin}}(X)\) consists of sets of the form
	\[
		[U_\alpha : \alpha \in \Lambda]
		:= \left\{ F \in [X]^{<\aleph_0} : \left( F \subseteq \bigcup \{U_\alpha : \alpha \in \Lambda\} \right) \wedge \left[ \forall \alpha \in \Lambda \ \left(F \cap U_\alpha \neq \varnothing\right) \right] \right\},
	\]
	where \(\# \Lambda < \aleph_0\) and \(U_\alpha \in \mathscr T_X\) for each \(\alpha \in \Lambda\).
	When the finite set of open sets is represented as \(\{U_1, \ldots, U_n\}\), we will use the more traditional notation
	of \([U_1,\ldots, U_n]\) for \([U_j : 1 \leq j \leq n]\).
\end{definition}
The hyperspaces \(\mathcal P_{\mathrm{fin}}(X)\) and \(\mathcal F[X]\) are related in an obvious way.
\begin{proposition}
	The topology on \([X]^{<\aleph_0}\) corresponding to \(\mathcal P_{\mathrm{fin}}(X)\) is, in general, strictly coarser than the one
	corresponding to \(\mathcal F[X]\).
	In other words, the identity map \(\mathcal F[X] \to \mathcal P_{\mathrm{fin}}(X)\) is continuous.
\end{proposition}

One potential benefit to studying \(\mathcal P_{\mathrm{fin}}(X)\) is that, as is well known, it contains a homeomorphic copy of \(X\).
\begin{proposition}
	The map \(x \mapsto \{x\}\), \(X \to \mathcal P_{\mathrm{fin}}(X)\), is a homeomorphism onto its range.
\end{proposition}
By way of contrast,
\begin{proposition}
	The image of the map \(x \mapsto \{x\}\), \(X \to \mathcal F[X]\), is closed and relatively compact.
	Hence, the map \(x \mapsto \{x\}\), \(X \to \mathcal F[X]\), is continuous if and only if \(X\) is discrete.
\end{proposition}
Another potential benefit to studying \(\mathcal P_{\mathrm{fin}}(X)\) is that it can capture some selective properties that are witnessed
in all finite powers of \(X\); for example, see particular results from \cite{CHVietoris,C_HurewiczVietoris}.

\subsection{Chasing Subsequences} \label{sec:Subsequences}

Throughout this note, we will be interested in subsequences of a given sequence of a space \(X\).
\begin{definition}
	Let \[\mathbb S = \left\{ \phi \in \mathbb N^{\mathbb N} : \forall m, n \in \mathbb N\ \left( m<n \implies \phi(m) < \phi(n) \right) \right\},\]
	the set of strictly increasing functions \(\mathbb N \to \mathbb N\).
\end{definition}
Note here that, for a space \(X\) and a sequence \(s \in X^{\mathbb N}\), any subsequence of \(s\) can be captured
by \(s \circ \phi\) for suitable \(\phi \in \mathbb S\).

We note the following fact which can be proved with a routine argument.
\begin{proposition}
	\(\mathbb S\) is an uncountable Polish semigroup with the property that, for any space \(X\),
	\(\mathbb S\) acts continuously on \(X^{\mathbb N}\) by function composition; in fact,
	the map \(\langle \phi, s \rangle \mapsto s \circ \phi\), \(\mathbb N^{\mathbb N} \times X^{\mathbb N} \to X^{\mathbb N}\),
	is continuous.
\end{proposition}
\begin{definition}
	We now let
	\[\mathrm{SDS}_X = \{ s \in X^{\mathbb N} : \forall \phi \in \mathbb S\ (s \circ \phi \not\in \mathrm{CS}_X) \},\]
	the set of all \emph{strongly divergent sequences of \(X\)}.
\end{definition}
Note that \(s \in \mathrm{SDS}_X\) if and only if every subsequence of \(s\) is divergent.
This notation also allows us to represent the property of being SHD in terms of selection principles.
\begin{remark}
	A space \(X\) is SHD if and only if \(X \models \vec{\mathsf S}_1(\mathscr T, \mathrm{SDS})\).
\end{remark}

Recall that \(S_\infty\), the permutation group of \(\mathbb N\), is an uncountable Polish group under the topology of pointwise convergence.
The group \(S_\infty\), like \(\mathbb S\), also acts naturally on \(X^{\mathbb N}\), where the action is \(\langle \phi, s \rangle \mapsto s \circ \phi\),
\(S_\infty \times X^{\mathbb N} \to X^{\mathbb N}\).
Note then that both \(\mathrm{CS}_X\) and \(\mathrm{SDS}_X\) are invariant under the actions of both \(\mathbb S\) and \(S_\infty\).

\begin{definition}
	We define
	\[\mathrm{wSD}_X = \{ s \in X^{\mathbb N} : \exists \phi \in \mathbb S\ (s \circ \phi \in \mathrm{SDS}_X) \},\]
	the set of \emph{weakly strongly divergent} sequences of \(X\).
\end{definition}
Note that
\[\mathrm{SDS}_X \subseteq \mathrm{wSD}_X,\]
and, in general, the inclusion is proper.

As above, we can express the notion of being WSHD in terms of selection principles.
\begin{remark}
	A space \(X\) is WSHD if and only if \(X \models \vec{\mathsf S}_1(\mathscr T, \mathrm{wSD})\).
\end{remark}

Since \(\mathbb S\) is an uncountable Polish space, and hence a Baire space, we can incorporate the notion of category.
\begin{definition}
	Define the set of
	\emph{generically strongly divergent sequences of \(X\)} to be
	\[\mathrm{gSDS}_X = \{ s \in X^{\mathbb N} : \forall^\ast \phi \in \mathbb S\ (s \circ \phi \in \mathrm{SDS}_X) \}\]
	where the expression \((\forall^\ast \phi \in \mathbb S)\ P(\phi)\) means that \(\{ \phi \in \mathbb S : P(\phi) \}\) is a comeager
	subset of \(\mathbb S\).
\end{definition}
Note then that
\[\mathrm{SDS}_X \subseteq \mathrm{gSDS}_X \subseteq \mathrm{wSD}_X.\]
\begin{definition}
	We can alternatively define the set of \emph{generically divergent sequences of \(X\)} to be
	\[\mathrm{gDS}_X = \{ s \in X^{\mathbb N} : \forall^\ast \phi \in \mathbb S\ (s \circ \phi \not\in \mathrm{CS}_X) \}.\]
\end{definition}
Note that \[\mathrm{gSDS}_X \subseteq \mathrm{gDS}_X.\]

\begin{question}
	Is it possible that \[\mathrm{SDS} = \mathrm{gSDS} = \mathrm{gDS}?\]
\end{question}

We note here that a similar definition could be made relative to one's favorite Borel measure on \(\mathbb S\).
In fact, such a notion can be defined relative to any nontrivial filter on \(\mathbb S\).
However, we will not entertain this direction explicitly here.
Instead, we turn our attention to notions of asymptotic density since they can capture a notion of \emph{most} naturals
and generate natural subsemigroups of \(\mathbb S\), which will become relevant, for example, in Lemma \ref{lem:PushingSDS}.

\subsection{Asymptotic Density and Statistical Convergence} \label{sec:Asymptotic}

Asymptotic density, also referred to as natural density, is a common measure used mainly in number theory to gauge how \emph{large} a given
subset of \(\mathbb N\) is.
For a thorough treatment of asymptotic density, see \cite{Niven}.
\begin{definition} \label{def:AsymptoticDensity}
	For \(A \subseteq \mathbb N\),
	define the \emph{asymptotic density} of \(A\) to be
	\[\delta(A) = \lim_{n \to \infty} \frac{\#(A \cap [1,n])}{n},\]
	provided the limit exists.
	
	For each \(\phi \in \mathbb S\), define
	\[\delta(\phi) = \delta(\{\phi(k) : k \in \mathbb N \}),\]
	whenever the latter is defined.
\end{definition}
\begin{remark} \label{rmk:BasicAsymptoticThing}
	As noted in \cite{Niven}, 
	\[\delta(\phi) = \lim_{n\to\infty} \frac{n}{\phi(n)},\]
	whenever \(\delta(\phi)\) is defined.
\end{remark}
\begin{definition}
For each \(\alpha \in [0,1]\), let \[\mathbf D_\alpha = \delta^{-1}(\alpha) = \{ \phi \in \mathbb S : \delta(\phi) = \alpha \}.\]
We will refer to elements of \(\mathbf D_1\) as being \emph{\(\delta_1\)-dense}.
\end{definition}

Di Maio and Ko{\v{c}}inac introduced in \cite{DiMaioKocinacStatistical} two natural notions of convergence in general topological spaces
based on asymptotic density.
\begin{definition}[{\cite{DiMaioKocinacStatistical}}]
	Let \(X\) be a space.
	Then a sequence \(s \in X^{\mathbb N}\) is said to \emph{statistically converge to \(x \in X\)} if,
	for every \(U \in \mathcal  N_{X,x}\),
	\(\delta\left( \{n \in \mathbb N : s_n \not\in U \} \right) = 0\);
	equivalently, \(s\) statistically converges to \(x\) if
	\[(\forall U \in \mathcal N_{X,x})(\exists \phi \in \mathbf D_1)(\forall n \in \mathbb N)\ s \circ \phi(n) \in U.\]
\end{definition}
\begin{definition}[{\cite{DiMaioKocinacStatistical}}]
	Let \(X\) be a space.
	A sequence \(s \in X^{\mathbb N}\) is said to \emph{\(s^\ast\)-converge to \(x \in X\)} if there exists \(\phi \in \mathbf D_1\) such that
	\(s \circ \phi \to x\).
\end{definition}
As recorded in \cite[Lemma 2.1]{DiMaioKocinacStatistical}, if a sequence \(s^\ast\)-converges to a point, then it also statistically converges to that point.
Also, the notions coincide for first-countable spaces, as witnessed by \cite[Theorem 2.2]{DiMaioKocinacStatistical}.
However, we are left wondering:
\begin{question}
	Are the notions of statistical convergence and \(s^\ast\)-convergence equivalent for all spaces?
\end{question}

As an immediate consequence of \cite[Theorem 6]{Niven}, we have
\begin{proposition} \label{prop:DenseSubsemigroup}
	For any \(\alpha, \beta \in [0,1]\),
	\(\mathbf D_\alpha \circ \mathbf D_\beta \subseteq \mathbf D_{\alpha\beta}\).
	In particular, \(\mathbf D_1\) is a subsemigroup of \(\mathbb S\).
\end{proposition}

We will also consider here a slight modification to the notion of asymptotic density, the
notion of \(\alpha\)-density due to \cite{Colak}, which arises as a natural gauge for the size of subsets of naturals
related to \(\lambda\)-statistical convergence introduced in \cite{LambdaConvergence}.
As can be seen, the \(\alpha\)-density offers a way to distinguish subsets of the naturals of zero asymptotic density.
There are a variety of other adaptations of the \(\lambda\)-statistical convergence notion, such as in \cite{ColakBektas},
but we will limit our attention to \(\alpha\)-density and related notions here.
\begin{definition}[{\cite[Def. 2.1]{Colak}}]
	For \(A \subseteq \mathbb N\) and \(\alpha \in (0,1]\), define
	the \emph{\(\alpha\)-density} of \(A \subseteq \mathbb N\) to be
	\[\delta_\alpha(A) = \lim_{n \to \infty} \frac{\#(A \cap [1,n])}{n^\alpha},\]
	provided the limit exists.
	
	For each \(\phi \in \mathbb S\), define
	\[\delta_\alpha(\phi) = \delta_\alpha(\{\phi(k) : k \in \mathbb N \}),\]
	whenever the latter is defined.
\end{definition}
Note that \(\delta_1\) agrees with the standard notion of asymptotic density.
\begin{remark}
	Given \(\phi \in \mathbb S\) such that \(\delta_\alpha(\phi)\) exists, note that
	\begin{align*}
		\delta_\alpha(\phi)
		&= \lim_{n \to \infty} \frac{\#(\{\phi(k) : k \in \mathbb N\} \cap [1,n])}{n^\alpha}\\
		&= \lim_{n \to \infty} \frac{\#(\{\phi(k) : k \in \mathbb N\} \cap [1,\phi(n)])}{\phi(n)^\alpha}\\
		&= \lim_{n \to \infty} \frac{n}{\phi(n)^\alpha}.
	\end{align*}
\end{remark}
We introduce the following as an analog to \(\mathbf D_1\).
\begin{definition}
	For \(\phi \in \mathbb S\), we will let \(C_\phi = \mathbb N \setminus \{ \phi(n) : n \in \mathbb N \}\).
	Then, for \(\alpha \in (0,1]\), we will define
	\[\mathbf D_\alpha^\ast = \{ \phi \in \mathbb S : \delta_\alpha ( C_\phi ) = 0 \}.\]
	We will refer to elements of \(\mathbf D_\alpha^\ast\) as being \emph{\(\delta_\alpha\)-dense}.
\end{definition}
Observe that \(\mathbf D_1 = \mathbf D_1^\ast\), and so the particular property of being \(\delta_1\)-dense
is unambiguously defined.
\begin{remark}
	Note that
	\begin{align*}
		\#(C_\phi \cap [1,\phi(n)])
		&= \#([1,\phi(n)] \setminus \{ \phi(k) : k \in \mathbb N \})\\
		&= \phi(n) - n.
	\end{align*}
	Hence, whenever \(\delta_\alpha(C_\phi)\) exists,
	\begin{align*}
		\delta_\alpha(C_\phi) &= \lim_{n\to\infty} \frac{\#(C_\phi \cap [1,n])}{n^\alpha}\\
		&= \lim_{n\to\infty} \frac{\#(C_\phi \cap [1,\phi(n)])}{\phi(n)^\alpha}\\
		&= \lim_{n\to\infty} \frac{\phi(n) - n}{\phi(n)^\alpha}.
	\end{align*}
\end{remark}

We offer here a bit of a refinement of \cite[Lemma 2.2]{Colak}, which states that \(\delta_\beta(A) \leq \delta_\alpha(A)\)
whenever \(0 < \alpha < \beta \leq 1\).
\begin{lemma} \label{lemma:InitialBounding}
	Suppose \(0 < \alpha < \beta \leq 1\) and \(\phi \in \mathbb S\).
	If \(\left\{\dfrac{n}{\phi(n)^\alpha} : n \in \mathbb N \right\}\) is bounded,
	then \(\delta_\beta(\phi) = 0\).
	Consequently, if \(\delta_\alpha(\phi)\) is defined and finite, then \(\delta_\beta(\phi) = 0\).
\end{lemma}
\begin{proof}
	Let \(M > 0\) be such that \(\dfrac{n}{\phi(n)^\alpha} \leq M\) for all \(n \in \mathbb N\).
	Note then that
	\[\frac{n}{\phi(n)^\beta} = \frac{n}{\phi(n)^\alpha} \cdot \frac{1}{\phi(n)^{\beta - \alpha}} \leq \frac{M}{\phi(n)^{\beta - \alpha}}.\]
	It follows that
	\[\lim_{n\to\infty} \frac{n}{\phi(n)^\beta} = 0.\]
	
	To finish the proof, note that
	\[
		0 \leq \liminf_{n\to\infty} \frac{\#(\{\phi(k) : k \in \mathbb N\} \cap [1,n])}{n^\beta}
		\leq \limsup_{n\to\infty} \frac{\#(\{\phi(k) : k \in \mathbb N\} \cap [1,n])}{n^\beta}
		= \limsup_{n\to\infty} \frac{n}{\phi(n)^\beta} = 0.
	\] \qedhere
\end{proof}
\begin{corollary}
	For \(0 < \alpha < \beta \leq 1\), \(\mathbf D_\alpha^\ast \subseteq \mathbf D_\beta^\ast \subseteq \mathbf D_1\).
\end{corollary}

\begin{lemma} \label{lemma:Bounding}
	Suppose \(\phi, \psi \in \mathbf D_\alpha^\ast\) for \(\alpha \in (0,1]\).
	Then \(\left\{ \dfrac{\phi(\psi(n+1))}{\phi(\psi(n))} : n \in \mathbb N \right\}\) is bounded.
\end{lemma}
\begin{proof}
	Since \(\phi, \psi \in \mathbf D_\alpha^\ast\), \(\delta_\alpha(C_\phi) = \delta_\alpha(C_\psi) = 0\).
	Hence, by Lemma \ref{lemma:InitialBounding}, \(\delta(C_\phi) = \delta(C_\psi) = 0\).
	It follows that \(\delta(\phi) = \delta(\psi) = 1\).
	To complete the proof, observe that
	\[
		\lim_{n\to\infty} \frac{\phi(\psi(n+1))}{\phi(\psi(n))}
		= \lim_{n\to\infty} \frac{\phi(\psi(n+1))}{\psi(n+1)} \cdot \frac{\psi(n)}{\phi(\psi(n))} \cdot \frac{\psi(n+1)}{n+1} \cdot \frac{n}{\psi(n)}
		= 1.
	\]
	Hence, \(\left\{ \dfrac{\phi(\psi(n+1))}{\phi(\psi(n))} : n \in \mathbb N \right\}\) is bounded.
\end{proof}

We now prove a partial analogue to Proposition \ref{prop:DenseSubsemigroup}.
\begin{theorem} \label{thm:Dalpha}
	For any \(\alpha \in (0,1]\), \(\mathbf D_\alpha^\ast\) is a subsemigroup of \(\mathbb S\).
\end{theorem}
\begin{proof}
	Suppose \(\phi, \psi \in \mathbf D_\alpha^\ast\).
	If \(C_{\phi \circ \psi}\) is finite, there is nothing to show since \(\delta_\alpha(C_{\phi \circ \psi}) = 0\).
	So suppose that \(C_{\phi \circ \psi}\) is infinite and note that \(\left\{ \dfrac{\phi(\psi(n+1))^\alpha}{\phi(\psi(n))^\alpha} : n \in \mathbb N \right\}\) is bounded
	by Lemma \ref{lemma:Bounding}.
	Let \(M > 0\) be such that \(\dfrac{\phi(\psi(n+1))^\alpha}{\phi(\psi(n))^\alpha} \leq M\) for all \(n \in \mathbb N\).
	
	Note that
	\begin{align*}
		\frac{\#(C_{\phi \circ \psi} \cap [1,\phi(\psi(n))])}{\phi(\psi(n))^\alpha}
		&= \frac{\phi(\psi(n)) - n}{\phi(\psi(n))^\alpha}\\
		&= \frac{\phi(\psi(n)) - \psi(n) + \psi(n) - n}{\phi(\psi(n))^\alpha}\\
		&= \frac{\phi(\psi(n)) - \psi(n)}{\phi(\psi(n))^\alpha} +  \frac{\psi(n) - n}{\phi(\psi(n))^\alpha}\\
		&= \frac{\phi(\psi(n)) - \psi(n)}{\phi(\psi(n))^\alpha} +  \frac{\psi(n) - n}{\psi(n)^\alpha} \cdot \frac{\psi(n)^\alpha}{\phi(\psi(n))^\alpha}\\
		&\leq \frac{\phi(\psi(n)) - \psi(n)}{\phi(\psi(n))^\alpha} +  \frac{\psi(n) - n}{\psi(n)^\alpha}.
	\end{align*}
	Since \(\phi, \psi \in \mathbf D_\alpha^\ast\),
	\[\lim_{n\to\infty} \frac{\phi(\psi(n)) - \psi(n)}{\phi(\psi(n))^\alpha} = 0\]
	and
	\[\lim_{n \to \infty} \frac{\psi(n) - n}{\psi(n)^\alpha} = 0.\]
	In particular,
	\[\lim_{n\to\infty} \frac{\#(C_{\phi \circ \psi} \cap [1,\phi(\psi(n))])}{\phi(\psi(n))^\alpha} = 0.\]
	
	Now, to show that 
	\[\lim_{n\to\infty} \frac{\#(C_{\phi \circ \psi} \cap [1,n])}{n^\alpha} = 0,\]
	let \(m_n \in \mathbb N\), for \(n > \phi(\psi(1))\), be maximal such that \(\phi(\psi(m_n)) < n\).
	Note that \(\phi(\psi(m_n)) < n \leq \phi(\psi(m_n+1))\) and so
	\begin{align*}
		\frac{\#(C_{\phi \circ \psi} \cap [1, n])}{n^\alpha}
		&\leq \frac{\#(C_{\phi \circ \psi} \cap [1, \phi(\psi(m_n+1))])}{\phi(\psi(m_n))^\alpha} \cdot \frac{\phi(\psi(m_n+1))^\alpha}{\phi(\psi(m_n+1))^\alpha}\\
		&= \frac{\#(C_{\phi \circ \psi} \cap [1, \phi(\psi(m_n+1))])}{\phi(\psi(m_n+1))^\alpha} \cdot \frac{\phi(\psi(m_n+1))^\alpha}{\phi(\psi(m_n))^\alpha}\\
		&\leq \frac{\#(C_{\phi \circ \psi} \cap [1, \phi(\psi(m_n+1))])}{\phi(\psi(m_n+1))^\alpha} \cdot M.
	\end{align*}
	It follows that
	\[0 \leq \lim_{n\to\infty} \frac{\#(C_{\phi \circ \psi} \cap [1, n])}{n^\alpha}
	\leq \lim_{n\to\infty} \frac{\#(C_{\phi \circ \psi} \cap [1, \phi(\psi(m_n+1))])}{\phi(\psi(m_n+1))^\alpha} \cdot M = 0,\]
	concluding the proof.
\end{proof}

We also note here a particular subsemigroup of \(\mathbb S\) that lives within each \(\mathbf D_\alpha^\ast\) which will be used to characterize
a class of locally finite sequences in Proposition \ref{prop:LocallyFinite}.
\begin{definition}
	We let
	\[
		\mathcal C := \{ \phi \in \mathbb S : \# C_\phi < \aleph_0 \},
	\]
	the collection of strictly increasing functions \(\mathbb N \to \mathbb N\)
	that have a cofinite range.
\end{definition}
The following lemma can be proved with a routine argument.
\begin{lemma} \label{lem:CofiniteSequencesThing}
	For any \(\phi \in \mathcal C\), there exist \(n, M\in \mathbb N\) such that \(\phi(n+k) = M+k\)
	for each \(k \in \omega\).
\end{lemma}
\begin{proposition} \label{prop:CofiniteSubsemigroup}
	The set \(\mathcal C\) is a subsemigroup of \(\mathbb S\).
\end{proposition}
\begin{proof}
	Let \(\phi , \psi \in \mathcal C\) and, by Lemma \ref{lem:CofiniteSequencesThing}, let \(n_\phi, n_\psi, M_\phi, M_\psi \in \mathbb N\)
	be such that \(\phi(n_\phi+k) = M_\phi + k\) and \(\psi(n_\psi + k) = M_\psi + k\) for all \(k \in \omega\).
	We now show that \([M_\phi+M_\psi , \infty) \subseteq \mathrm{range}(\phi \circ \psi)\).	
	To accomplish this, simply note that, for any \(k \in \omega\),
	\begin{align*}
		\phi(\psi(n_\phi+n_\psi+k))
		&= \phi(\psi(n_\psi + (n_\phi + k)))\\
		&= \phi(M_\psi + n_\phi + k)\\
		&= \phi(n_\phi + (M_\psi + k))\\
		&= M_\phi + M_\psi + k.
	\end{align*}
	This completes the proof.
\end{proof}

We relativize the notion of weakly strongly divergent sequences to the class of subsemigroups of \(\mathbb S\) consisting of \(\mathbf D_\alpha^\ast\).
\begin{definition}
	For a space \(X\) and \(\alpha \in (0,1]\), we define
	\[\mathrm{wSD}^\alpha_X = \{ s \in X^{\mathbb N} : \exists \phi \in \mathbf D_\alpha^\ast\ (s \circ \phi \in \mathrm{SDS}_X) \}.\]
\end{definition}
Note that, for any space \(X\) and \(0 < \alpha \leq \beta \leq 1\), we have that
\[\mathrm{SDS}_X \subseteq \mathrm{wSD}^\alpha_X \subseteq \mathrm{wSD}^\beta_X \subseteq \mathrm{wSD}^1_X \subseteq \mathrm{wSD}_X,\]
and the inclusions are, in general, proper.
Indeed, given any \(s \in \mathrm{SDS}_X\) where \(X\) is an infinite Hausdorff space, one can weave a constant sequence in with \(s\)
to form a sequence \(\varsigma\) with the property that \(\varsigma \circ \phi = s\) for some specified \(\phi \in \mathbb S\) with an
infinite complement \(C_\phi\), and \(\varsigma \circ \psi\) is constant where \(\psi\in \mathbb S\) is an enumeration of \(C_\phi\).

\begin{definition}
	For a space \(X\) and \(\alpha \in (0,1]\), we define
	\[\mathrm{SDS}_X^\alpha = \{ s \in X^{\mathbb N} : \forall \phi \in \mathbf D_\alpha^\ast\ (s \circ \phi \not\in \mathrm{CS}_X) \}.\]
	In other words, \(\mathrm{SDS}^\alpha_X\) consists of all sequences of \(X\) for which no \(\delta_\alpha\)-dense subsequence converges.
\end{definition}
Note that, for any space \(X\) and \(0 < \alpha \leq \beta \leq 1\), we have that
\[\mathrm{SDS}_X \subseteq \mathrm{SDS}^1_X \subseteq \mathrm{SDS}_X^\beta \subseteq \mathrm{SDS}_X^\alpha.\]

Inspired by the notions of statistical convergence and \(s^\ast\)-convergence, we propose the following relativized versions.
\begin{definition}
	Let \(X\) be a space and \(\alpha \in (0,1]\).
	Then a sequence \(s \in X^{\mathbb N}\) is said to \emph{\(\alpha\)-statistically converge to \(x \in X\)} if,
	for every \(U \in \mathcal N_{X,x}\),
	\(\delta_\alpha\left( \{n \in \mathbb N : s_n \not\in U \} \right) = 0\);
	equivalently, \(s\) statistically converges to \(x\) if
	\[(\forall U \in \mathcal N_{X,x})(\exists \phi \in \mathbf D_\alpha^\ast)(\forall n \in \mathbb N)\ s \circ \phi(n) \in U.\]
\end{definition}
\begin{definition}
	Let \(X\) be a space and \(\alpha \in (0,1]\).
	A sequence \(s \in X^{\mathbb N}\) is said to \emph{\(s_\alpha^\ast\)-converge to \(x \in X\)} if there exists \(\phi \in \mathbf D_\alpha^\ast\) such that
	\(s \circ \phi \to x\).
\end{definition}
Note here that the set of \(s^\ast_\alpha\)-convergent sequences of a space \(X\) constitutes precisely the complement of \(\mathrm{SDS}^\alpha_X\).

As with statistical convergence and \(s^\ast\)-convergence, we ask:
\begin{question}
	Are the notions of \(\alpha\)-statistical convergence and \(s_\alpha^\ast\)-convergence equivalent for all \(\alpha \in (0,1]\) and all spaces?
\end{question}

We remark here that there are notions of convergence with respect to ideals of \(\mathbb N\), as in \cite{Lahiri2005}.
However, we will not consider them explicitly in this work.

We end this section with some basic remarks on the behavior of \(\alpha\)-density under composition.
\begin{example}
	The sequences of positive finite \(\alpha\)-density are not generally closed under composition.
	Consider \(\phi \in \mathbb S\) defined by \(\phi(n) = n^2\).
	Note that \(\phi \circ \phi(n) = n^4\).
	Hence, \(\delta_{1/2}(\phi) = 1\) and \(\delta_{1/2}(\phi \circ \phi) = 0\).
\end{example}
\begin{example}
	Sequences of infinite \(\alpha\)-density may compose to form a sequence of zero \(\alpha\)-density.
	Consider \(\phi \in \mathbb S\) defined by \(\phi(n) = n^2\).
	Note that \(\phi \circ \phi(n) = n^4\).
	Hence, \(\delta_{1/3}(\phi) = \infty\) and \(\delta_{1/3}(\phi \circ \phi) = 0\).
\end{example}
\begin{example}
	There are sequences of infinite \(\alpha\)-density such that their corresponding complements also have infinite \(\alpha\)-density.
	Consider \(\phi : \mathbb N \to \mathbb N\) defined by \(\phi(n) = 2n\) and note that \(\delta_{1/2}(\phi) = \infty\) and
	\(\delta_{1/2}(C_\phi) = \infty\).
\end{example}
As a final note, we remark that, unlike sequence convergence, these notions of statistical convergence are not invariant under the action
of \(S_\infty\).
\begin{example}
	Let \(s : \mathbb N \to \mathbb R\) be defined by
	\[
		s_n = \begin{cases} n, & \exists k \in \mathbb N\ (n = k^2)\\
		1/n, & \text{otherwise} \end{cases}
	\]
	Since the set of squares has asymptotic density zero in \(\mathbb N\), \(s\) statistically converges to \(0\); in fact,
	\(s\) \(s^\ast\)-converges to \(0\).
	
	Let \(A = \{ n^2 : n \in \mathbb N \}\) and note that \(\mathbb N \setminus A\) is also an infinite set.
	So we can construct a bijection \(\phi : \mathbb N \to \mathbb N\) such that \(\phi[A] = \mathbb N \setminus A\)
	and \(\phi[\mathbb N \setminus A] = A\).
	Note then that \(s \circ \phi\) fails to statistically converge to \(0\) since \(\delta(\{ n \in \mathbb N : s \circ \phi(n) \not\in (-1,1) \}) = 1\).
\end{example}

\subsection{When is winning on a subsequence enough?}

To stay in the context of sequences, we define, for a space \(X\),
\[\vec{\mathcal O}_X = \{ U \in \mathscr T_X^{\mathbb N} : (\forall x \in X)(\exists n \in \mathbb N)\ x \in U_n \}\]
and
\[\vec{\Gamma}_X = \{ U \in \mathscr T_X^{\mathbb N} : (\forall x \in X)(\exists n \in \mathbb N)(\forall m \geq n)\ x \in U_m \}.\]
Note that these are sequential versions of \(\mathcal O_X\), the set of all open covers of \(X\), and \(\Gamma_X\), the set of all
\(\gamma\)-covers of \(X\).
For more on commonly used cover types, like \(\gamma\)-covers, with relevant references, see \cite{KocinacSelectedResults}.

Now, for a set \(X\), \(\mathcal B \subseteq X^{\mathbb N}\), and \(\mathcal E \subseteq \mathbb S\),
we define
\[w \mathcal B^{\mathcal E} = \{ s \in X^{\mathbb N} : \exists \phi \in \mathcal E\ (s \circ \phi \in \mathcal B) \}\]
and
\[s \mathcal B^{\mathcal E} = \{ s \in X^{\mathbb N} : \forall \phi \in \mathcal E\ (s \circ \phi \in \mathcal B) \}.\]
When \(\mathcal E = \mathbb S\), we will suppress the superscript.

Immediately, we see that
\[
	\vec{\mathsf G}_1(\mathcal O_X, \vec{\mathcal O}_X)
	\rightleftarrows \vec{\mathsf G}_1(\mathcal O_X, w \vec{\mathcal O}_X).
\]

Though the following is known, we provide a proof for the convenience of the reader.
\begin{proposition}
	As topological operators, \(\vec{\Gamma} = s \vec{\mathcal O}\).
	Hence, for any space \(X\),
	\[
		\vec{\mathsf G}_1(\mathcal O_X, \vec{\Gamma}_X)
		\rightleftarrows \vec{\mathsf G}_1(\mathcal O_X, s \vec{\mathcal O}_X).
	\]
\end{proposition}
\begin{proof}
	Let \(X\) be a space.
	First, we show that \(\vec{\Gamma}_X \subseteq s \vec{\mathcal O}_X\).
	So let \(\langle U_n : n \in \mathbb N \rangle \in \vec{\Gamma}_X\) and let \(\phi \in \mathbb S\).
	We need to show that \(\langle U_{\phi(n)} : n \in \mathbb N \rangle \in \vec{\mathcal O}_X\).
	So let \(x \in X\) and \(n \in \mathbb N\) be such that, for all \(m \geq n\), \(x \in U_m\).
	Certainly, there is some \(m \in \mathbb N\) with \(\phi(m) \geq n\), so \(x \in U_{\phi(m)}\).
	That is, \(\langle U_{\phi(n)} : n \in \mathbb N \rangle \in \vec{\mathcal O}_X\).
	
	Now suppose \(\langle U_n : n \in \mathbb N \rangle \not\in \vec{\Gamma}\).
	Then we can let \(x \in X\) be such that, for every \(n \in \mathbb N\), there exists \(m \geq n\) with \(x \not\in U_m\).
	One can then recursively define \(\phi \in \mathbb S\) such that \(x \not\in U_{\phi(n)}\) for every \(n \in \mathbb N\).
	That is, \(\langle U_{\phi(n)} : n \in \mathbb N \rangle \not\in \vec{\mathcal O}_X\).
\end{proof}

It is also immediate that, for Gruenhage's converging \(W\)-game \(\mathsf G_1(\mathcal N_{X,x} , \neg \Gamma_{X,x})\) (see \cite{Gruenhage1976}), where \(\neg \mathcal A\) for a class \(\mathcal A\) denotes the complement of \(\mathcal A\), we have that
\[
	\mathsf G_1(\mathcal N_{X,x} , \neg \Gamma_{X,x})
	\rightleftarrows \mathsf G_1(\mathcal N_{X,x} , w(\neg \Gamma_{X,x})).
\]

Note that, using this notation, \(w\mathrm{SDS}_X = \mathrm{wSD}_X\).
Question 4 of \cite{RemarksOnSHD} can thus be generalized to a question schema.
\begin{question}
	Consider sets \(\mathcal A\), \(\mathcal B \subseteq \left( \bigcup \mathcal A \right)^{\mathbb N}\), and \(\mathcal E \subseteq \mathbb S\).
	When is it true that
	\[\vec{\mathsf S}_1(\mathcal A, w \mathcal B^{\mathcal E}) \implies \vec{\mathsf S}_1(\mathcal A, \mathcal B)?\]
\end{question}
We can extend this question to the context of game inequalities.
\begin{question}
	Consider sets \(\mathcal A\), \(\mathcal B \subseteq \left( \bigcup \mathcal A \right)^{\mathbb N}\), and \(\mathcal E \subseteq \mathbb S\).
	When is it true that
	\[\vec{\mathsf G}_1(\mathcal A, w \mathcal B^{\mathcal E}) \leq_{\mathrm{II}}^+ \vec{\mathsf G}_1(\mathcal A, \mathcal B)?\]
\end{question}
The general question here relates to P2's ability to control how much of \(\mathbb N\) they need to win.
For example, if there is a sequence of \(\mathcal A\) on which P2 can guarantee a winning play on a \(\delta_\alpha\)-dense
subsequence, can P2 \emph{always} guarantee a win on a \(\delta_\alpha\)-dense subsequence?

In relation to these questions, one can ask about the particular case for the Gruenhage clustering \(W\)-game \(\mathsf G_1(\mathcal N_{X,x}, \neg \Omega_{X,x})\)
(see \cite{Gruenhage1976}), where \(\Omega_{X,x} = \{ A \subseteq X : x \in \mathrm{cl}(A) \}\), which we phrase as follows.
\begin{question}
	For a space \(X\) and \(x \in X\), let \[\vec{\Omega}_{X,x} = \left\{ s \in X^{\mathbb N} : x \in \mathrm{cl}\{s_n:n\in\mathbb N\} \right\}.\]
	Is it true that \[\vec{\mathsf G}_1(\mathscr T_X, w(\neg \vec{\Omega}_{X,x})) \leq_{\mathrm{II}}^+ \vec{\mathsf G}_1(\mathscr T_X, \neg \vec{\Omega}_{X,x})?\]
\end{question}

\section{Results on Selective Divergence Properties}

\subsection{Some commentary on discrete selectivity}

We first note that allowing for finite selections in the closed discrete selection game doesn't add anything.
\begin{proposition} \label{prop:CDSingleFiniteSame}
	For any space \(X\),
	\[\mathsf G_1(\mathscr T_X, \mathrm{CD}_X)
	\rightleftarrows \mathsf G_{\mathrm{fin}}(\mathscr T_X, \mathrm{CD}_X).\]
\end{proposition}
\begin{proof}
	We need only show that
	\[\mathsf G_{\mathrm{fin}}(\mathscr T_X , \mathrm{CD}_X)
		\leq_{\mathrm{II}}^+ \mathsf G_1(\mathscr T_X, \mathrm{CD}_X).\]
	This follows from the fact that, if \(\bigcup \{ F_n : n \in \mathbb N \} \in \mathrm{CD}_X\)
	and \(x_n\in F_n\) for each \(n \in \mathbb N\), then \(\{x_n : n \in \mathbb N \} \subseteq \bigcup \{ F_n : n \in \mathbb N\}\) and hence,
	\(\{x_n:n\in\mathbb N\} \in \mathrm{CD}_X\).
\end{proof}

We list a few properties that are incompatible with being discretely selective.
Though Propositions \ref{prop:InfiniteHausdorffCountablyCompact} and \ref{prop:CountableLocalPiBase}
and Corollary \ref{cor:SequentiallyCompact} are surely known, we record proofs here for the convenience of the reader.
\begin{proposition} \label{prop:InfiniteHausdorffCountablyCompact}
	No infinite Hausdorff countably compact space is discretely selective.
\end{proposition}
\begin{proof}
	Suppose \(X\) is an infinite Hausdorff space which is discretely selective.
	By a routine argument\footnote{See, for example, \url{https://math.stackexchange.com/questions/547517}},
	we can let \(\langle U_n : n \in \mathbb N \rangle\) be a sequence of nonempty
	open sets such that \(\{U_n : n \in \mathbb N\}\) is pairwise disjoint.
	Since \(X\) is discretely selective, we can choose \(x_n \in U_n\) for each \(n \in \mathbb N\)
	to be such that \(\{x_n:n\in\mathbb N\}\) is closed and relatively discrete.
	So, for each \(n \in \mathbb N\), let \(V_n \in \mathscr T_X\) be such that \(V_n \cap \{x_j:j\in\mathbb N\} = \{x_n\}\)
	and consider the cover
	\[\mathscr V = \{ V_n : n \in \mathbb N \} \cup \{ X \setminus \{x_n:n\in\mathbb N\} \}.\]
	Note that \(\mathscr V\) is a countable open cover of \(X\) with no finite subcover.
	Hence, \(X\) is not countably compact.
\end{proof}
Consequently, the property of being discretely selective doesn't generally transfer to the {\v{C}}ech-Stone compactification
of a Tychonoff space.
\begin{example}
	The discrete space \(\omega\) is discretely selective but \(\beta \omega\) is not discretely selective.
\end{example}

Since every sequentially compact space is countably compact, we obtain the following immediate corollary.
\begin{corollary} \label{cor:SequentiallyCompact}
	No infinite Hausdorff sequentially compact space is discretely selective.
\end{corollary}
Note that any finite discrete space is an example of a discretely selective Hausdorff space which is sequentially compact.
\begin{example} \label{ex:CechStoneOmega}
	The space \((\beta \omega)^{\omega_1}\) is SHD but not discretely selective.
	Indeed, since \(\beta\omega\) is not sequentially compact (see \cite{Counterexamples}\footnote{See also \url{https://dantopology.wordpress.com/2012/10/01/stone-cech-compactification-of-the-integers-basic-facts/}}),
	\cite[Corollary 1]{RemarksOnSHD} guarantees that \((\beta \omega)^{\omega_1}\)
	is SHD.
	As an infinite compact Hausdorff space, \((\beta\omega)^{\omega_1}\) is not discretely selective.
\end{example}
In fact, the SHD property, even in the context of topological groups, does not imply the discretely selective property.
\begin{example} \label{ex:CirclePower}
	There is a connected Hausdorff group which is SHD but not discretely selective.
	Consider the usual circle group, \(\mathbb T = \{ z \in \mathbb C : |z| = 1 \}\); that is, the complex numbers of modulus \(1\) with complex multiplication.
	Note that \(\mathbb T\) is a connected compact Hausdorff group.
	Now, \([0,1]^{\mathfrak c}\) embeds into \(\mathbb T^{\mathfrak c}\) as a closed subspace.
	As \([0,1]^{\mathfrak c}\) is not sequentially compact (see \cite[Space 105]{Counterexamples}), \(\mathbb T^{\mathfrak c}\)
	is not sequentially compact.
	Note then that \((\mathbb T^{\mathfrak c})^{\omega_1} \simeq \mathbb T^{\mathfrak c}\) is SHD by \cite[Corollary 1]{RemarksOnSHD}.
	Since \(\mathbb T^{\mathfrak c}\) is an infinite compact Hausdorff space, it is not discretely selective.
\end{example}
For the next example, 
recall that the \emph{splitting number} \(\mathfrak s\) is the smallest cardinality of a subset \(\mathbf S\) of \([\omega]^{\aleph_0}\)
that has the property that, for every \(A \in [\omega]^{\aleph_0}\), there exists \(S \in \mathbf S\) such that
both \(A \cap S\) and \(A \setminus S\) are infinite.
For more on this cardinal, see \cite{Halbeisen}.
\begin{example} \label{ex:2splitting}
	The space \(2^{\mathfrak s}\) is a zero-dimensional compact Hausdorff group which is SHD but not discretely selective.
	\(2^{\mathfrak s}\) is SHD by \cite[Theorem 2.1]{BellaSpadaro} and it is not discretely selective as an infinite compact space.
\end{example}
\begin{example}
	Any discrete group is an example of a topological group which is discretely selective but not SHD.
\end{example}
\begin{question}
	Is there a Hausdorff group with no isolated points that is discretely selective but not SHD?
\end{question}
Recall that a \emph{local \(\pi\)-base} at a point \(x\) of a space \(X\) is a collection \(\mathcal B\)
of open subsets of \(X\) such that, for any \(U \in \mathcal N_{X,x}\), there exists \(V \in \mathcal B\)
with \(V \subseteq U\).
\begin{proposition} \label{prop:CountableLocalPiBase}
	Suppose \(X\) is a space and that \(x \in X\) is such that \(\{x\}\) is closed in \(X\), \(x\) is not an isolated point,
	and there is a countable local \(\pi\)-base at \(x\).
	Then \(X\) is not discretely selective.
	
	In particular, no first-countable \(T_1\) space without isolated points is discretely selective.
\end{proposition}
\begin{proof}
	Let \(\{U_n : n \in \mathbb N\}\) be a local \(\pi\)-base at \(x\) and consider the sequence \(\langle U_n \setminus \{x\} : n \in \mathbb N \rangle\)
	of nonempty open subsets of \(X\).
	Let \(y_n \in U_n \setminus \{x\}\) for each \(n \in \mathbb N\) and notice that
	\(x \in \mathrm{cl}_X\{y_n:n\in\mathbb N\}\).
	Hence, \(X\) is not discretely selective.
\end{proof}
Note however that any discrete space is an example of a \(T_1\) first-countable space which is discretely selective.

\begin{proposition}
	For spaces \(X\) and \(Y\), if \(\{ x_n : n \in \mathbb N \} \in \mathrm{CD}_X\) and \(\{ y_n : n \in \mathbb N \} \in \mathrm{CD}_Y\),
	then \(\{ \langle x_n, y_n \rangle : n \in \mathbb N \} \in \mathrm{CD}_{X \times Y}\).
	Consequently, for any \(\mathbf \Pi \in \left\{ \mathrm{II}\underset{\mathrm{mark}}{\uparrow}, \mathrm{II}\uparrow, \mathrm{I}\not\uparrow,
	\mathrm{I}\underset{\mathrm{pre}}{\not\uparrow}, \mathrm{I}\underset{\mathrm{cnst}}{\not\uparrow} \right\}\),
	\[
		\left( \mathbf \Pi \mathsf{G}_1(\mathscr T_X, \mathrm{CD}_X) \wedge \mathbf \Pi \mathsf{G}_1(\mathscr T_Y, \mathrm{CD}_Y) \right)
		\implies \mathbf \Pi \mathsf{G}_1(\mathscr T_{X \times Y} , \mathrm{CD}_{X \times Y}).
	\]
	In particular, the property of being discretely selective is finitely productive.
\end{proposition}
\begin{proof}
	Suppose that \(\{x_n:n\in\mathbb N\} \in \mathrm{CD}_X\) and \(\{y_n : n \in \mathbb N \} \in \mathrm{CD}_Y\).
	We show that \(\{ \langle x_n, y_n \rangle : n \in \mathbb N \} \in \mathrm{CD}_{X \times Y}\).
	
	To see that \(\{ \langle x_n,y_n \rangle : n \in \mathbb N\}\) is closed, let \(\langle x,y \rangle \not\in \{ \langle x_n,y_n \rangle : n \in \mathbb N \}\).
	We proceed by cases.
	
	Suppose \(x \not\in \{ x_n :n \in \mathbb N\}\) and let \(W = X \setminus \{ x_n : n \in \mathbb N \}\).
	Note that \(\langle x,y \rangle \in W \times Y\) and that \(W \times Y \cap \{ \langle x_n,y_n \rangle : n \in \mathbb N \} = \varnothing\).
	
	Suppose \(y \not\in \{ y_n : n \in \mathbb N \}\).
	Then, for \(W = Y \setminus \{ y_n : n \in \mathbb N \}\), just as above, \(X \times W\) is a neighborhood of \(\langle x,y \rangle\)
	that is disjoint from \(\{ \langle x_n, y_n \rangle : n \in \mathbb N \}\).
	
	Otherwise, suppose \(x \in \{ x_n : n \in \mathbb N \}\) and \(y \in \{y_n:n\in\mathbb N\}\).
	Let \(j,k \in \mathbb N\) be such that \(x = x_j\) and \(y = y_k\).
	Note that \(j \neq k\) since \(\langle x,y \rangle \not\in \{ \langle x_n,y_n \rangle : n \in \mathbb N \}\).
	Let \(W_x \in \mathscr T_X\) and \(W_y \in \mathscr T_Y\) be such that \(W_x \cap \{ x_n : n \in \mathbb N \} = \{ x_j \}\)
	and \(W_y \cap \{ y_n :n \in \mathbb N \} = \{ y_k \}\).
	Then \(\langle x,y \rangle \in W_x \times W_y\).
	Note that \(W_x \times W_y \cap \{ \langle x_n , y_n \rangle : n \in \mathbb N\} = \varnothing\).
	
	For relative discreteness, let \(j \in \mathbb N\).
	Then let \(W_x \in \mathscr T_X\) and \(W_y \in \mathscr T_Y\) be such that \(W_x \cap \{ x_n : n \in \mathbb N \} = \{ x_j \}\)
	and \(W_y \cap \{ y_n : n \in \mathbb N \} = \{ y_j \}\).
	Note that \(W_x \times W_y \cap \{ \langle x_n,y_n \rangle : n \in \mathbb N \}  = \{ \langle x_j,y_j \rangle \}\).
	
	Hence, \(\{ \langle x_n, y_n \rangle : n \in \mathbb N \} \in \mathrm{CD}_{X \times Y}\).
	
	For the game-theoretic implication, fix choices \(U_W \in \mathscr T_X\) and \(V_W \in \mathscr T_Y\) for each \(W \in \mathscr T_{X \times Y}\)
	so that \(U_W \times V_W \subseteq W\).
	Note that choices \(x \in U_W\) and \(y \in V_W\) satisfy \(\langle x, y \rangle \in U_W \times V_W \subseteq W\).
	Verifying each case proposed for the game-theoretic implication is thus routine and left to the reader.
\end{proof}

However, discrete selectivity is not preserved by arbitrary products.
\begin{example}
	The discrete space \(\omega\) is discretely selective, but \(\omega^\omega\) is not discretely selective.
\end{example}

As can be seen from \cite[Theorem 23]{ClontzHolshouser} and \cite[Theorem 13]{RemarksOnSHD},
it is sometimes possible to do a selection process to get a sequence of distinct selections.
With this in mind, we introduce the following notion.
\begin{definition}
We define
\[\mathrm{CD}_X^\ast = \{ s \in X^{\mathbb N} : (\forall n , m \in \mathbb N)[n \neq m \implies s_n \neq s_m] \wedge \{ s_n : n \in \mathbb N \} \in \mathrm{CD}_X \}.\]
\end{definition}
Note that this is just capturing injective sequences of \(X\) such that their range is closed and relatively discrete in \(X\).
\begin{definition} \label{def:InjDiscSel}
	We will say a space \(X\) is \emph{injectively discretely selective} if \(X \models \vec{\mathsf S}_1(\mathscr T, \mathrm{CD}^\ast)\).
\end{definition}
\begin{remark} \label{remark:ObviousCD}
	Observe that, for any space \(X\),
	\[\vec{\mathsf G}_1(\mathscr T_X, \mathrm{CD}_X^\ast) \leq_{\mathrm{II}}^+ \mathsf G_1(\mathscr T_X, \mathrm{CD}_X).\]
\end{remark}
These two games are, in general, however, not equivalent.
\begin{example}
	Any discrete space \(X\) has the properties that
	\[\mathrm{II} \underset{\mathrm{mark}}{\uparrow} \mathsf G_1(\mathscr T_X, \mathrm{CD}_X) \text{ and }
	\mathrm{I} \underset{\mathrm{pre}}{\uparrow} \vec{\mathsf G}_1(\mathscr T_X, \mathrm{CD}_X^\ast).\]
	In particular, this means that
	\[
		\mathrm{II} \underset{\mathrm{mark}}{\uparrow} \mathsf G_1(\mathscr T_X, \mathrm{CD}_X) \notimplies
		\mathrm{II} \underset{\mathrm{mark}}{\uparrow} \vec{\mathsf G}_1(\mathscr T_X, \mathrm{CD}_X^\ast)
	\]
	and
	\[
		\mathrm{I} \underset{\mathrm{pre}}{\not\uparrow} \mathsf G_1(\mathscr T_X, \mathrm{CD}_X) \notimplies
		\mathrm{I} \underset{\mathrm{pre}}{\not\uparrow} \vec{\mathsf G}_1(\mathscr T_X, \mathrm{CD}_X^\ast).
	\]
	As such, any discrete space is an example of a space which is discretely selective but not injectively discretely selective.
\end{example}

However, when \(X\) is \(T_1\) and every open set is infinite, we obtain equivalence at the level of full-information strategies.
\begin{theorem}
	Let \(X\) be a \(T_1\) space in which every open set is infinite.
	Then
	\[
		\mathrm{II} \uparrow \vec{\mathsf G}_1(\mathscr T_X, \mathrm{CD}_X^\ast)
		\iff \mathrm{II} \uparrow \mathsf G_1(\mathscr T_X, \mathrm{CD}_X)
	\]
	and
	\[
		\mathrm{I} \uparrow \vec{\mathsf G}_1(\mathscr T_X, \mathrm{CD}_X^\ast)
		\iff \mathrm{I} \uparrow \mathsf G_1(\mathscr T_X, \mathrm{CD}_X).
	\]
\end{theorem}
\begin{proof}
	From Remark \ref{remark:ObviousCD}, we have the implications
	\[
		\mathrm{II} \uparrow \vec{\mathsf G}_1(\mathscr T_X, \mathrm{CD}_X^\ast)
		\implies \mathrm{II} \uparrow \mathsf G_1(\mathscr T_X, \mathrm{CD}_X)
	\]
	and
	\[
		\mathrm{I} \not\uparrow \vec{\mathsf G}_1(\mathscr T_X, \mathrm{CD}_X^\ast)
		\implies \mathrm{I} \not\uparrow \mathsf G_1(\mathscr T_X, \mathrm{CD}_X).
	\]
	Hence, to complete the proof, we show the reverse implications.
	
	Suppose P2 has a winning strategy \(\tau\) in \(\mathsf G_1(\mathscr T_X, \mathrm{CD}_X)\).
	We recursively define a strategy \(\tilde\tau\) for P2 in \(\vec{\mathsf G}_1(\mathscr T_X, \mathrm{CD}_X^\ast)\)
	as follows.
	
	For \(n \in \mathbb N\), consider a sequence \(\langle U_\ell : 1 \leq \ell \leq n \rangle\) of open subsets of \(X\).
	Let \(V_1 = U_1\) and \(F_1 = \varnothing\).
	For a fixed \(k < n\), suppose we've defined \(V_\ell\) and \(F_\ell\) for each \(1 \leq \ell \leq k\).
	Let \[F_{k+1} = F_k \cup \left\{\tau(\langle V_\ell : 1 \leq \ell \leq k \rangle)\right\}\] and \(V_{k+1} = U_{k+1} \setminus F_{k+1}\).
	Note that \(V_{k+1}\) is a nonempty open set since \(X\) is \(T_1\) and \(U_{k+1}\) is infinite.
	
	This recursively defines a sequence of open sets \(\langle V_\ell : 1 \leq \ell \leq n \rangle\).
	Then let
	\[\tilde\tau(\langle U_\ell : 1 \leq \ell \leq n \rangle) = \tau(\langle V_\ell : 1 \leq \ell \leq n \rangle).\]
	This defines the strategy \(\tilde\tau\).
	
	To see that \(\tilde\tau\) is winning, let \(\langle U_n : n \in \mathbb N \rangle\) be a sequence of open
	subsets of \(X\) and note that \(\tilde\tau\) selects a sequence \(\langle x_n : n \in \mathbb N \rangle\)
	where \(x_n \in U_n\) for each \(n \in \mathbb N\) and the map \(n \mapsto x_n\), \(\mathbb N \to X\),
	is injective by construction.
	Also, since the sequence \(\langle x_n : n \in \mathbb N \rangle\) was chosen according to \(\tau\),
	\(\{ x_n : n \in \mathbb N \}\) is closed and relatively discrete.
	
	Now suppose that P1 does not have any winning strategy in the closed discrete selection game on \(X\) and let \(\sigma\)
	be a strategy for P1 in \(\vec{\mathsf G}_1(\mathscr T_X, \mathrm{CD}_X^\ast)\).
	We define a strategy \(\tilde\sigma\) in the closed discrete selection game on \(X\) in the following way.
	First, we define \(U_1 = \tilde\sigma(\varnothing) = \sigma(\varnothing)\).
	For \(n \in \mathbb N\), suppose we have defined \(\langle U_\ell : 1 \leq \ell \leq n\rangle\)
	and \(\langle x_\ell : 1 \leq \ell < n \rangle\).
	For any \(x_n \in U_n\), let
	\[U_{n+1} = \tilde\sigma(\langle x_\ell : 1 \leq \ell \leq n \rangle) 
	= \sigma(\langle x_\ell : 1 \leq \ell \leq n \rangle) \setminus \{ x_\ell : 1 \leq \ell \leq n \}.\]
	This recursively defines \(\tilde\sigma\).
	
	By hypothesis, \(\tilde\sigma\) is not a winning strategy, so there exists
	\(\langle x_n : n \in \mathbb N \rangle\) where \(x_n \in \tilde\sigma(\langle x_\ell : 1 \leq \ell < n \rangle)\) for every \(n \in \mathbb N\)
	with the additional property that \(\{ x_n : n \in \mathbb N \} \in \mathrm{CD}_X\).
	Since the mapping \(n \mapsto x_n\), \(\mathbb N \to X\), is an injection,
	we see that \(\langle x_n : n \in \mathbb N \rangle \in \mathrm{CD}^\ast_X\).
	That is, \(\sigma\) is not a winning strategy for P1 in \(\vec{\mathsf G}_1(\mathscr T_X, \mathrm{CD}_X^\ast)\).
\end{proof}

The next lemma effectively asserts that, if you have a countable closed and relatively discrete subset of \(X\) and you group together finitely many
elements at a time, then you form a countable closed discrete subset of \(\mathcal P_{\mathrm{fin}}(X)\).
\begin{lemma} \label{lem:GroupCD}
	Suppose \(s \in X^{\mathbb N}\) is such that \(\{ s_n : n \in \mathbb N \} \in \mathrm{CD}_X\)
	and suppose \(\phi \in \mathbb S\).
	Let \(F_1 = \{ s_j : 1 \leq j \leq \phi(1) \}\) and, for \(n \in \mathbb N\), \(F_{n+1} =  \{ s_j : \phi(n) < j \leq \phi(n+1) \}\).
	Then \(\{ F_n : n \in \mathbb N \} \in \mathrm{CD}_{\mathcal P_{\mathrm{fin}}(X)}\).
\end{lemma}
\begin{proof}
	We start by showing \(\{ F_n : n \in \mathbb N \}\) is a closed subset of \(\mathcal P_{\mathrm{fin}}(X)\).
	Consider \(F \not\in \{ F_n : n \in \mathbb N \}\).
	We proceed by cases.
	
	If \(F \setminus \{ s_n : n \in \mathbb N \} \neq \varnothing\), fix \(x \in F \setminus \{ s_n : n \in \mathbb N \}\).
	Since \(\{s_n : n \in \mathbb N\}\) is closed in \(X\), \(U := X \setminus \{s_n:n\in\mathbb N\}\) is an open set
	and \(x \in U\).
	Note then that
	\[F \in [U,X] \subseteq \mathcal P_{\mathrm{fin}}(X) \setminus \{ F_n : n \in \mathbb N \}.\]
	
	If \(F \subseteq \{ s_n : n \in \mathbb N \}\),
	for each \(x \in F\), we can find an open set \(U_x\) of \(X\)
	with \(U_x \cap \{ s_n : n \in \mathbb N \} = \{x\}\) since \(\{ s_n : n \in \mathbb N \}\) is relatively discrete.
	Now note that
	\[F \in [U_x : x \in F] \subseteq \mathcal P_{\mathrm{fin}}(X) \setminus \{ F_n : n \in \mathbb N \}.\]
	
	Hence, \(\{F_n:n\in\mathbb N\}\) is closed in \(\mathcal P_{\mathrm{fin}}(X)\).
	
	We finish by showing that \(\{F_n:n\in\mathbb N\}\) is relatively discrete.
	By the relative discreteness of \(\{s_n:n\in\mathbb N\}\), we can choose an open set \(U_n\) of \(X\)
	with \(U_n \cap \{s_n:n\in\mathbb N\} = \{s_n\}\) for each \(n\in\mathbb N\).
	Note then that
	\[[U_j : 1 \leq j \leq \phi(1)] \cap \{ F_k : k \in \mathbb N \} = \{ F_1 \}\]
	and, for each \(n \in \mathbb N\),
	\[[U_j : \phi(n) < j \leq \phi(n+1)] \cap \{ F_k : k \in \mathbb N \} = \{ F_{n+1} \}.\] \qedhere
\end{proof}

We now offer an equivalent game to the closed discrete game on \(X\) in terms of \(\mathcal P_{\mathrm{fin}}(X)\).
\begin{theorem} \label{thm:PfinCD}
	For a space \(X\), let \[\mathrm{SCD}_{\mathcal P_{\mathrm{fin}}(X)}
	= \left\{ \mathbf F \in \mathrm{CD}_{\mathcal P_{\mathrm{fin}}(X)} : \bigcup \mathbf F \in \mathrm{CD}_X \right\}.\]	
	Then \[\mathsf G_1(\mathscr T_X, \mathrm{CD}_X) \rightleftarrows \mathsf G_1(\mathscr T_{\mathcal P_{\mathrm{fin}}(X)} , \mathrm{SCD}_{\mathcal P_{\mathrm{fin}}(X)}).\]
	Consequently, for any space \(X\),
	\begin{equation} \label{eqn:CDLEQ}
		\mathsf G_1(\mathscr T_X, \mathrm{CD}_X)
		\leq^+_{\mathrm{II}} \mathsf G_1(\mathscr T_{\mathcal P_{\mathrm{fin}}(X)} , \mathrm{CD}_{\mathcal P_{\mathrm{fin}}(X)})
	\end{equation}
	and
	\begin{equation} \label{eqn:ICDLEQ}
		\vec{\mathsf G}_1(\mathscr T_X, \mathrm{CD}^\ast_X)
		\leq^+_{\mathrm{II}} \vec{\mathsf G}_1(\mathscr T_{\mathcal P_{\mathrm{fin}}(X)} , \mathrm{CD}^\ast_{\mathcal P_{\mathrm{fin}}(X)}).
	\end{equation}
\end{theorem}
\begin{proof}
	We start by showing that
	\begin{equation} \label{eqn:MainInequality}
		\mathsf G_1(\mathscr T_X, \mathrm{CD}_X) \leq_{\mathrm{II}}^+ \mathsf G_1(\mathscr T_{\mathcal P_{\mathrm{fin}}(X)} , \mathrm{SCD}_{\mathcal P_{\mathrm{fin}}(X)}).
	\end{equation}
	Since \(\mathrm{SCD}_{\mathcal P_{\mathrm{fin}}(X)} \subseteq \mathrm{CD}_{\mathcal P_{\mathrm{fin}}(X)}\),
	this will establish \eqref{eqn:CDLEQ}.
	For \eqref{eqn:ICDLEQ}, simply note that, throughout the proof of \eqref{eqn:MainInequality},
	if the selection in the closed discrete selection game on \(X\) can be chosen to form an injective function,
	then the grouped selection for \(\mathsf G_1(\mathscr T_{\mathcal P_{\mathrm{fin}}(X)} , \mathrm{SCD}_{\mathcal P_{\mathrm{fin}}(X)})\)
	can also be taken to form an injective function.
	
	Assume that P2 has a winning Markov strategy \(\tau : \mathscr T_X \times \mathbb N \to X\) in the closed discrete selection game on \(X\).
	Fix a bijection \(\beta : \mathbb N^2 \to \mathbb N\) and also a choice of basic neighborhood (in a fixed order)
	to be contained in any open subset of \(\mathcal P_{\mathrm{fin}}(X)\).
	We define \(\tilde\tau : \mathscr T_{\mathcal P_{\mathrm{fin}}(X)} \times \mathbb N \to \mathcal P_{\mathrm{fin}}(X)\) as follows.
	Consider \(W \in \mathscr T_{\mathcal P_{\mathrm{fin}}(X)}\), \(n \in \mathbb N\),
	and the chosen subset \([U_{1} , \ldots , U_{m}]\) of \(W\).
	Then define
	\[\tilde\tau(W,n) = \{ \tau(U_{j},\beta(n,j)) : 1 \leq j \leq m \}.\]
	Note that \(\tilde\tau(W,n) \in W\).
	
	Now we need to show that \(\tilde\tau\) is winning for P2 in the modified closed discrete selection game on \(\mathcal P_{\mathrm{fin}}(X)\).
	So consider a sequence \(\langle W_n : n \in \mathbb N \rangle\) of open subsets of \(\mathcal P_{\mathrm{fin}}(X)\) and let
	\([U_{n,1}, \ldots, U_{n,m_n}]\) represent the chosen basic neighborhood contained in \(W_n\).
	For every \(k > m_n\), let \(U_{n,k} = X\) and consider
	\[\{ \tau(U_{n,k} , \beta(n,k)) : n,k \in \mathbb N \}.\]
	Since \(\tau\) is winning for P2 in the closed discrete selection game on \(X\),
	\[\{ \tau(U_{n,k} , \beta(n,k)) : n,k \in \mathbb N \} \in \mathrm{CD}_X.\]
	Since any subset of a closed and relatively discrete set is closed and relatively discrete, note that
	\[\{ \tau(U_{n,k} , \beta(n,k)) : n \in \mathbb N, 1 \leq k \leq m_n \} \in \mathrm{CD}_X.\]
	Then, by Lemma \ref{lem:GroupCD}, \(\{ \tilde\tau(W_n,n) : n \in \mathbb N \} \in \mathrm{CD}_{\mathcal P_{\mathrm{fin}}(X)}\).
	Moreover, since \[\bigcup \{ \tilde\tau(W_n,n) : n \in \mathbb N \} = \{ \tau(U_{n,k} , \beta(n,k)) : n \in \mathbb N, 1 \leq k \leq m_n \},\]
	we see that \(\{ \tilde\tau(W_n,n) : n \in \mathbb N \} \in \mathrm{SCD}_{\mathcal P_{\mathrm{fin}}(X)}\).
	
	For the remaining implications, excluding the case for constant strategies,
	we continue with the fixed basic neighborhood assignment to each open subset of
	\(\mathcal P_{\mathrm{fin}}(X)\), and assume without loss of generality that P1 is just playing those basic neighborhoods.
	We translate an initial play
	\[\langle [U_{1,1}, \ldots, U_{1,m_1}] , [U_{2,1}, \ldots, U_{2,m_2}] , \ldots, [U_{n,1} , \ldots, U_{n,m_n}] \rangle\]
	by P1 in the modified closed discrete selection game on \(\mathcal P_{\mathrm{fin}}(X)\) to an initial play
	\[\langle U_{1,1} , \ldots, U_{1,m_1}, U_{2,1} , \ldots, U_{2,m_2} , \ldots , U_{n,1} , \ldots , U_{n,m_n} \rangle\]
	by P1 in the closed discrete selection game on \(X\).
	As P2 makes choices \(x_{j,k} \in U_{j,k}\), we translate these back to the game on \(\mathcal P_{\mathrm{fin}}(X)\)
	with \(F_j = \{ x_{j,k} : 1 \leq k \leq m_j \}\).
	
	If P2 has a winning strategy in \(\mathsf G_1(\mathscr T_X, \mathrm{CD}_X)\), then this translation scheme generates
	a winning strategy for P2 in \(\mathsf G_1(\mathscr T_{\mathcal P_{\mathrm{fin}}(X)} , \mathrm{SCD}_{\mathcal P_{\mathrm{fin}}(X)})\)
	by Lemma \ref{lem:GroupCD} and the observation that
	\(\bigcup_{j \in \mathbb N} F_j \in \mathrm{CD}_X\).
	
	If P1 fails to have a winning strategy in the closed discrete selection game on \(X\), consider a strategy \(\sigma\)
	for P1 in the modified closed discrete selection game on \(\mathcal P_{\mathrm{fin}}(X)\).
	This translation scheme generates a strategy \(\tilde\sigma\) for P1 in \(\mathsf G_1(\mathscr T_X, \mathrm{CD}_X)\).
	Since \(\tilde\sigma\) is not winning, there exists a selection that beats \(\tilde\sigma\), which, by Lemma \ref{lem:GroupCD},
	beats \(\sigma\).
	So P1 has no winning strategy in \(\mathsf G_1(\mathscr T_{\mathcal P_{\mathrm{fin}}(X)} , \mathrm{SCD}_{\mathcal P_{\mathrm{fin}}(X)})\).
	
	In the case that P1 doesn't have a pre-determined winning strategy in \(\mathsf G_1(\mathscr T_X, \mathrm{CD}_X)\), the fact that
	P1 also doesn't have a pre-determined winning strategy in \(\mathsf G_1(\mathscr T_{\mathcal P_{\mathrm{fin}}(X)} , \mathrm{SCD}_{\mathcal P_{\mathrm{fin}}(X)})\)
	follows in essentially the same way as it did above for the perfect-information strategy case.
	
	We now show that, if P1 doesn't have a constant winning strategy in \(\mathsf G_1(\mathscr T_X, \mathrm{CD}_X)\),
	then P1 doesn't have a constant winning strategy in \(\mathsf G_1(\mathscr T_{\mathcal P_{\mathrm{fin}}(X)} , \mathrm{SCD}_{\mathcal P_{\mathrm{fin}}(X)})\).
	We prove this by induction on the length of basic neighborhoods.

	First, suppose we are given \([U]\).
	Then we can pick \(\{x_n : n \in \mathbb N\} \subseteq U\) which is closed and relatively discrete in \(X\).
	Then \(\{ \{x_n\} : n \in \mathbb N \} \subseteq [U]\) can be seen to be a member of \(\mathrm{SCD}_{\mathcal P_{\mathrm{fin}}(X)}\)
	by Lemma \ref{lem:GroupCD}.
	
	For \(k \in \mathbb N\), suppose we have shown that, given any \([U_1, \ldots, U_k]\), we can find \(\{F_n:n\in\mathbb N\} \subseteq [U_1,\ldots, U_k]\)
	such that \(\{ F_n : n \in \mathbb N \} \in \mathrm{SCD}_{\mathcal P_{\mathrm{fin}}(X)}\).
	Consider \([U_1, \ldots, U_{k+1}]\).
	We proceed by cases.
	
	Suppose there is some \(j = 1, \dots, k\) such that \(U_{k+1} \cap U_j \neq \varnothing\).
	For each \(m = 1 , \ldots, k\), let \[W_m = \begin{cases} U_m, & m \neq j \\ U_{k+1} \cap U_m , & m = j \end{cases}.\]
	Note that
	\[[W_1, \ldots,  W_k] \subseteq [U_1, \ldots, U_{k+1}]\]
	and we can apply the inductive hypothesis on \([W_1, \ldots,  W_k]\), which finishes this case.
	
	Otherwise, suppose that \(U_{k+1} \cap U_j = \varnothing\) for each \(j = 1 , \ldots, k\).
	Apply the inductive hypothesis to produce \(\{G_n:n\in\mathbb N\} \subseteq [U_1, \ldots, U_k]\)
	such that \(\{G_n:n\in\mathbb N\} \in \mathrm{SCD}_{\mathcal P_{\mathrm{fin}}(X)}\).
	We can also apply the assumption that P1 doesn't have a constant winning strategy in \(\mathsf G_1(\mathscr T_X, \mathrm{CD}_X)\)
	to produce \(\{x_n:n\in\mathbb N\} \subseteq [U_{k+1}]\) such that \(\{x_n:n \in \mathbb N\} \in \mathrm{CD}_X\).
	Let \(F_n = G_n \cup \{x_n\}\) and note that \(F_n \in [U_1,\ldots, U_{k+1}]\) for each \(n \in \mathbb N\).
	We claim that \(\{F_n : n \in \mathbb N \} \in \mathrm{SCD}_{\mathcal P_{\mathrm{fin}}(X)}\).
	
	Note that, since \(\bigcup_{n\in\mathbb N} G_n \in \mathrm{CD}_X\) and \(\{x_n : n \in \mathbb N\} \in \mathrm{CD}_X\),
	\[\mathbf F := \bigcup_{n\in\mathbb N} F_n = \{x_n:n\in\mathbb N\} \cup \bigcup_{n\in\mathbb N} G_n\]
	is a closed subset of \(X\).
	For relative discreteness of \(\mathbf F\), let \(y \in \mathbf F\).
	
	If \(y \in \bigcup_{n\in\mathbb N} G_n\), we can find an open set \(V\) such that \(V \cap \bigcup_{n\in\mathbb N} G_n = \{y\}\).
	Note that \(y \in \tilde U := \bigcup_{m=1}^k U_m\) and that \(\{ x_n : n \in \mathbb N \} \cap \tilde U = \varnothing\).
	Hence, \(\left( V \cap \tilde U \right) \cap \mathbf F = \{y\}\).
	
	Otherwise, \(y \in \{x_n:n\in\mathbb N\}\).
	Then we can find an open set \(V\) such that \(V \cap \{x_n:n\in\mathbb N\} = \{y\}\).
	In a similar fashion as above, \(\left( V \cap U_{k+1} \right) \cap \mathbf F = \{y\}\).
	
	Consequently, as \(\mathbf F \in \mathrm{CD}_X\), Lemma \ref{lem:GroupCD} can be seen to guarantee that
	\(\{ F_n : n \in \mathbb N \} \in \mathrm{SCD}_{\mathcal P_{\mathrm{fin}}(X)}\).
	This concludes our demonstration of
	\eqref{eqn:MainInequality}.
	
	To finish the proof, we show that
	\[\mathsf G_1(\mathscr T_{\mathcal P_{\mathrm{fin}}(X)} , \mathrm{SCD}_{\mathcal P_{\mathrm{fin}}(X)})
	\leq_{\mathrm{II}}^+ \mathsf G_1(\mathscr T_X, \mathrm{CD}_X).\]
	This will follow, via Proposition \ref{prop:CDSingleFiniteSame}, once we establish that
	\[\mathsf G_1(\mathscr T_{\mathcal P_{\mathrm{fin}}(X)} , \mathrm{SCD}_{\mathcal P_{\mathrm{fin}}(X)})
	\leq_{\mathrm{II}}^+ \mathsf G_{\mathrm{fin}}(\mathscr T_X, \mathrm{CD}_X).\]
	Indeed, the mapping \(U \mapsto [U]\), \(\mathscr T_X \to \mathscr T_{\mathcal P_{\mathrm{fin}}(X)}\),
	easily guarantees that
	\[\mathsf G_1(\mathscr T_{\mathcal P_{\mathrm{fin}}(X)} , \mathrm{SCD}_{\mathcal P_{\mathrm{fin}}(X)})
	\leq_{\mathrm{II}}^+ \mathsf G_{\mathrm{fin}}(\mathscr T_X, \mathrm{CD}_X)\]
	since any corresponding selection \(\{F_n : n \in \mathbb N\} \in \mathrm{SCD}_{\mathcal P_{\mathrm{fin}}(X)}\)
	has the property that \(\bigcup \{ F_n : n \in \mathbb N \} \in \mathrm{CD}_X\).
\end{proof}
\begin{question} \label{question:PfinDS}
	Is it possible that, for all spaces \(X\),
	\[\mathsf G_1(\mathscr T_X, \mathrm{CD}_X) \rightleftarrows \mathsf G_1(\mathscr T_{\mathcal P_{\mathrm{fin}}(X)} , \mathrm{CD}_{\mathcal P_{\mathrm{fin}}(X)})?\]
\end{question}

One would hope that the techniques for proving Theorem \ref{thm:PfinCD} could be adapted to show an analogous result
for \(\mathcal F[X]\).
However, P1 can force P2 to contain particular finite sets in their selections, so P2 cannot control the properties of their selections
in the same way.
Alternatively, one may hope that simply appending select points, as in the proof of \cite[Theorem 2.6]{BellaSpadaro}, may work.
However, as we illustrate below, without greater care, such an approach cannot necessarily be applied.
For this, we consider two basic scenarios relative to \(\mathbb R\).

For each \(n \in \mathbb N\), let \(F_n = \left\{ \frac{1}{n} \right\}\) and consider the basic neighborhood \([F_n, (-1,1)]\) of
\(F_n\) in \(\mathcal F[X]\).
For \(n \in \mathbb N\), let \(x_n = 0\) and \(G_n = F_n \cup \{x_n\}\).
Note that \(G_n \in [F_n, (-1,1)]\) for each \(n \in \mathbb N\) and that \(\{x_n:n\in\mathbb N\} \in \mathrm{CD}_{\mathbb R}\).
However, \[\{0\} \in \left[\mathrm{cl}_{\mathcal F[X]} \bigcup_{n\in\mathbb N} G_n \right] \setminus \bigcup_{n\in\mathbb N} G_n.\]

One may suspect that the only issue above is the convergence of the sequence \(\left\langle \frac{1}{n} : n \in \mathbb N \right\rangle\) to \(0\),
and so simply avoiding any potential limit points of the \(F_n\) may remedy the situation.
However, this also doesn't necessarily work.
Indeed, define \(F_n = \left\{ 0, \frac{1}{n} \right\}\) and consider \([F_n , (-2,1)]\) for each \(n \in \mathbb N\).
Then we can let \(x_n = -1\) for each \(n \in \mathbb N\) and form \(G_n = F_n \cup \{-1\}\).
Note that \(-1\) is not a limit point of any selection from the \(F_n\).
Also note that \(\{x_n:n\in\mathbb N\} \in \mathrm{CD}_{\mathbb R}\).
Nevertheless, \(\langle G_n : n \in \mathbb N \rangle\) converges to \(\{-1,0\}\) in \(\mathcal F[\mathbb R]\)
and \(\{-1,0\} \not\in \{G_n : n \in \mathbb N\}\).

We note that neither \(\mathbb R\) nor \(\mathcal F[\mathbb R]\) are discretely selective, so the above considerations aren't counterexamples to the conjecture
that \(\mathcal F[X]\) is discretely selective whenever \(X\) is discretely selective.
However, we have not yet identified a general argument that would prove such a claim, and leave Question \ref{question:FCD} at the end of this section
for further investigation.

We now turn our attention to isolating some sufficient conditions for a space \(X\) which do guarantee that \(\mathcal F[X]\)
is discretely selective.
\begin{definition}
	We say that a space \(X\) is \emph{nowhere separable} if no nonempty open subset of \(X\) is separable.
\end{definition}
Quantifying only over open subsets of \(X\) is clearly motivated by the fact that any space has separable subspaces.
It is also immediate that nowhere separable spaces fail to be separable.
This implication does not reverse since any uncountable discrete space is neither separable nor nowhere separable.

\begin{theorem} \label{thm:SufficientFXDS}
	Let \(X\) be a space.
	\begin{enumerate}[label=(\roman*),ref=(\roman*)]
		\item \label{discretePixley}
		If \(X\) is discrete, then \(\mathcal F[X]\) is discrete.
		In particular, both \(X\) and \(\mathcal F[X]\) are discretely selective.
		\item \label{nowhereSeparable}
		If \(X\) is \(T_1\), injectively discretely selective, and nowhere separable, then \(\mathcal F[X]\) is discretely selective (by virtue of
		being injectively discretely selective).
	\end{enumerate}
\end{theorem}
\begin{proof}
	We need only show \ref{nowhereSeparable} since \ref{discretePixley} is immediate.
	So suppose \(X\) is \(T_1\), injectively discretely selective, and nowhere separable.
	Let \(\langle [F_n , U_n] : n \in \mathbb N \rangle\)
	be a sequence of basic open subsets of \(\mathcal F[X]\)
	and let \(\mathbf F = \mathrm{cl}_X \bigcup \{ F_n : n \in \mathbb N \}\).
	Since \(X\) is nowhere separable, \(U_n \setminus \mathbf F \neq \varnothing\) for each \(n \in \mathbb N\).
	So let \(x_n \in U_n \setminus \mathbf F\) be chosen in such a way that \(n \mapsto x_n\), \(\mathbb N \to X\),
	is injective and \(\{ x_n : n \in \mathbb N \} \in \mathrm{CD}_X\).
	Define \(G_n = F_n \cup \{x_n\}\) and note that \(G_n \in [F_n,U_n]\).
	We will show that \(\mathcal G := \{G_n : n \in \mathbb N\} \in \mathrm{CD}_{\mathcal F[X]}\).
	
	Before addressing the closed and relative discreteness of \(\mathcal G\), let \(V_n \in \mathscr T_X\)
	be such that \(V_n \cap \{x_k:k\in\mathbb N\} = \{x_n\}\) and \(V_n \subseteq X \setminus \mathbf F\)
	for each \(n \in \mathbb N\).
	
	We first show that \(\mathcal G\) is closed
	So let \(E \not\in \mathcal G\).
	We proceed by cases.
	
	If \(E \cap \{ x_n : n \in \mathbb N \} = \varnothing\), we can set \(W = X \setminus \{x_n:n\in\mathbb N\}\)
	and conclude that \([E,W] \cap \mathcal G = \varnothing\).
	
	Otherwise, \(E \cap \{ x_n : n \in \mathbb N \} \neq \varnothing\).
	Let \(\Lambda = \{ n \in \mathbb N : x_n \in E \}\).
	We now consider sub-cases.
	
	If it so happens that \(E = \{ x_n : n \in \Lambda \}\), then \(W := \bigcup \{ V_n : n \in \Lambda \}\)
	has the property that \([E,W] \cap \mathcal G = \varnothing\).
	
	Otherwise, \(E \setminus \{x_n : n \in \mathbb N\} \neq \varnothing\).
	So let \(\{y_1,\ldots, y_m\}\) be a faithful enumeration of \(E \setminus \{x_n:n\in\mathbb N\}\).
	For each \(j = 1,2,\ldots, m\), we can let \(W_j\) be an open subset of \(X\) with the property that
	\[W_j \cap \left[ E \cup \bigcup_{n = 1}^{\max \Lambda} G_n \right] = \{y_j\}.\]
	Then we set \[W = \bigcup \{ V_n : n \in \Lambda\} \cup \bigcup_{j=1}^m W_j\]
	and observe that \(E \subseteq W\).
	
	Note that, for any \(n > \max \Lambda\), \(E \not \subseteq G_n\) since \(x_{\max \Lambda} \in E\) and
	\(x_{\max \Lambda} \not\in G_n\).
	That is, \(G_n \not\in [E, W]\) for any \(n > \max \Lambda\).
	
	If \(\# \Lambda \geq 2\), we have that \(E \not\subseteq G_n\) for any \(n \in \mathbb N\).
	Indeed, let \(n \in \mathbb N\) be arbitrary and suppose \(x_n \in E\).
	Since \(\# \Lambda \geq 2\), we can let \(m \neq n\) be such that \(x_m \in E\).
	Note that \(x_m \not\in G_n\), so \(E \not\subseteq G_n\).
	
	So suppose \(\# \Lambda = 1\) and, for \(n \in \mathbb N\), suppose that \(E \subseteq G_n\).
	For the \(\lambda \in \Lambda\), \(x_\lambda \in E \subseteq G_n\), which means that \(n = \lambda\).
	Since \(E \neq G_n\), there must be some \(z \in G_n \setminus E\).
	In particular, \(z \in F_n\) and so we see that \(z \not\in W\).
	That is, \(G_n \not\in [E,W]\).
	
	Conclusively, \(\mathcal G\) is closed.
	
	We finish by showing that \(\mathcal G\) is relatively discrete.
	So let \(n \in \mathbb N\) be arbitrary and consider
	\(W_n = V_n \cup \left( X \setminus \{ x_k : k \in \mathbb N \} \right)\).
	Note that \(G_n \subseteq W_n\) since \(x_n \in V_n\) and \(F_n \subseteq X \setminus \{ x_k : k \in \mathbb N \}\).
	Moreover, by the injectivity of \(n \mapsto x_n\), \(\mathbb N \to X\), and the fact that \(V_n \cap \{ x_k : k \in \mathbb N \} = \{ x_n \}\),
	we see that \([G_n,W_n] \cap \mathcal G = \{ G_n \}\).
\end{proof}
The property of being nowhere separable is not enough to guarantee the property of being discretely selective.
\begin{example} \label{example:2kappa}
	The space \(2^\kappa\), for a cardinal \(\kappa \geq \mathfrak c^+\), is a nowhere separable topological group which is not discretely selective.
	The fact that it is nowhere separable follows from \cite[Exercise 2.3.F(c)]{Engelking} and the fact that every nonempty open
	subset of \(2^\kappa\) contains a homeomorphic copy of \(2^\kappa\).
\end{example}
However, there are examples of \(T_1\), injectively discretely selective, nowhere separable spaces.
\begin{example}
	The space \(\mathbb Z^\kappa\), for a cardinal \(\kappa \geq \mathfrak c^+\), is an example of a nowhere separable \(T_1\) topological group which is
	injectively discretely selective.
	In fact, \(\mathrm{II} \underset{\mathrm{mark}}{\uparrow} \vec{\mathsf{G}}_1(\mathscr T_{\mathbb Z^\kappa}, \mathrm{CD}_{\mathbb Z^\kappa}^\ast)\).
\end{example}
\begin{proof}
	The fact that \(\mathbb Z^\kappa\) is nowhere separable follows from a similar argument as used in Example \ref{example:2kappa}.
	For the purposes of this argument, note that
	\[[f;F] := \{ g \in \mathbb Z^\kappa : g \restriction_F = f \restriction_F \},\]
	where \(f \in \mathbb Z^\kappa\) and \(F \in [\kappa]^{<\aleph_0}\), generates a basis of \(\mathbb Z^\kappa\).
	
	To see that \(\mathrm{II} \underset{\mathrm{mark}}{\uparrow} \vec{\mathsf{G}}_1(\mathscr T_{\mathbb Z^\kappa}, \mathrm{CD}_{\mathbb Z^\kappa}^\ast)\), define
	\[\tau : \mathbb Z^\kappa \times [\kappa]^{<\aleph_0} \times \mathbb N \to \mathbb Z^\kappa\]
	by the rule
	\[\tau(f,F,n)(x) = \begin{cases} f(x), & x \in F\\ n, & x \in \kappa \setminus F \end{cases}\]
	Consider a sequence \(\langle [f_n ; F_n] : n \in \mathbb N \rangle\),
	where \(f_n \in \mathbb Z^\kappa\) and \(F_n \in [\kappa]^{<\aleph_0}\),
	of nonempty basic open neighborhoods of \(\mathbb Z^\kappa\).
	Since \(\kappa\) is uncountable, there is some \(x \in \kappa \setminus \bigcup_{n\in\mathbb N} F_n\).
	Note then that \(\tau(f_n,F_n,n)(x) = n\) for each \(n \in \mathbb N\), and so, for any \(h \in \mathbb Z^\kappa\),
	\[\#\{ n \in \mathbb N : \tau(f_n,F_n,n) \in [h; \{x\}] \} \leq 1.\]
	Hence, \(\langle \tau(f_n,F_n,n) : n \in \mathbb N \rangle \in \mathrm{CD}^\ast_{\mathbb Z^\kappa}\).
\end{proof}
We also note that, just as with the SHD property, \(\mathcal F[X]\) being discretely selective does not imply that \(X\) must be discretely selective.
To do this we note that the proof of \cite[Theorem 5]{RemarksOnSHD} is a proof of the following assertion.
\begin{proposition} \label{prop:UncountableDiscretelySelectiveHyper}
	For an uncountable cardinal \(\kappa\) and a dense subspace \(X\) of
	\[\Sigma(2,0,\kappa) := \left\{ f \in 2^\kappa : \#\{ \alpha \in \kappa : f(\alpha) \neq 0 \} \leq \aleph_0 \right\},\]
	viewed as a subspace of \(2^\kappa\),
	\(\mathcal F[X]\) is discretely selective.
\end{proposition}
\begin{proof}
	We need only observe that the family \(\{ G_n: n \in \mathbb N \}\) generated
	in the proof of \cite[Theorem 5]{RemarksOnSHD} is locally finite, hence closed and relatively discrete since \(\mathcal F[X]\) is a \(T_1\) space.
\end{proof}
\begin{example} \label{example:SigmaProduct}
	For an uncountable cardinal \(\kappa\), \(X := \Sigma(2,0,\kappa)\) is a space which is not discretely selective,
	but for which \(\mathcal F[X]\) is discretely selective.
	Indeed, it is known that \(\Sigma(2,0,\kappa)\) is sequentially compact for any cardinal \(\kappa\).
	To see this, consider any sequence \(\langle x_n : n \in \mathbb N \rangle\) of \(\Sigma(2,0,\kappa)\) and let
	\(A_n = \{ \alpha \in \kappa : x_n(\alpha) \neq 0 \}\) for each \(n \in \mathbb N\).
	Note that \(A := \bigcup \left\{A_n : n \in \mathbb N \right\}\) is a countable subset of \(A\).
	Then the projection \(2^\kappa \to 2^A\) witnesses the existence of a convergent subsequence of \(\langle x_n : n \in \mathbb N\rangle\),
	just as argued in \cite[Example 2.2]{BellaSpadaro}.
	Hence, \(X\) is not discretely selective.
	Then Proposition \ref{prop:UncountableDiscretelySelectiveHyper} applies to guarantee that \(\mathcal F[X]\)
	is discretely selective.
\end{example}

Note that, even with the hypotheses for \(X\) in Theorem \ref{thm:SufficientFXDS} \ref{nowhereSeparable},
we can't necessarily guarantee transferral of strategic information, contrasting against Theorems \ref{thm:PfinCD} and \ref{thm:HyperSDS}.
\begin{question} \label{question:FCD}
	Are any of the implications constituting the expression
	\[\mathsf G_1(\mathscr T_X, \mathrm{CD}_X) \leq_{\mathrm{II}}^+ \mathsf G_1(\mathscr T_{\mathcal F[X]} , \mathrm{CD}_{\mathcal F[X]})\]
	true for all spaces \(X\)?
	In particular, is \(\mathcal F[X]\) discretely selective whenever \(X\) is discretely selective?
\end{question}

\subsection{Remarks on variations of the SHD property}

We first note that closed, relatively discrete subsets and sequences that are non-convergent on some class of subsequences transfer
from \(\mathcal P_{\mathrm{fin}}(X)\) to \(\mathcal F[X]\).
\begin{proposition} \label{prop:pushingCD}
	If \(\mathbf F \in \mathrm{CD}_{\mathcal P_{\mathrm{fin}}(X)}\), then \(\mathbf F \in \mathrm{CD}_{\mathcal F[X]}\).
\end{proposition}
\begin{proof}
	Suppose \(\mathbf F \in \mathrm{CD}_{\mathcal P_{\mathrm{fin}}(X)}\).
	Since \(\mathcal P_{\mathrm{fin}}(X)\) has a coarser topology than \(\mathcal F[X]\), \(\mathbf F\) is immediately seen to be
	closed in \(\mathcal F[X]\).
	To see that \(\mathbf F\) is relatively discrete in \(\mathcal F[X]\), let \(F \in \mathbf F\).
	By relative discreteness in \(\mathcal P_{\mathrm{fin}}(X)\), we can find a basic neighborhood
	\([U_1,\ldots, U_n]\) of \(F\) such that
	\[\mathbf F \cap [U_1, \ldots, U_n] = \{F\}.\]
	Let \(W = \bigcup_{j=1}^n U_j\) and consider the basic neighborhood \([F,W]\) of \(F\)
	in \(\mathcal F[X]\).
	Suppose \(G \in \mathbf F\) is such that \(G \in \mathbf F \cap [F,W]\).
	Since \(F \subseteq G\), we can see that
	\[\varnothing \neq F \cap U_j \subseteq G \cap U_j\]
	for each \(j = 1,\ldots, n\).
	Moreover, \(G \subseteq \bigcup_{j=1}^n U_j\).
	Hence, \(G \in [U_1, \ldots , U_n]\) and so \(F = G\).
\end{proof}
Note, however, that \(\{ \{x\} : x \in \mathbb R \}\) is a closed and relatively discrete subspace of \(\mathcal F[\mathbb R]\),
whereas \(x \mapsto \{x\}\), \(\mathbb R \to \mathcal P_{\mathrm{fin}}(\mathbb R)\), is a homeomorphism onto its range, and so \(\{ \{x\} : x \in \mathbb R \}\)
is not relatively discrete in \(\mathcal P_{\mathrm{fin}}(\mathbb R)\).
Hence, Proposition \ref{prop:pushingCD}, in general, does not reverse.
\begin{proposition} \label{prop:PushingConvergence}
	For a space \(X\), let \(F \in ([X]^{<\aleph_0})^{\mathbb N}\).
	Then
	\[
		\{ \phi \in \mathbb S : F \circ \phi \in \mathrm{CS}_{\mathcal F[X]} \}
		\subseteq \{ \phi \in \mathbb S : F \circ \phi \in \mathrm{CS}_{\mathcal P_{\mathrm{fin}}(X)} \}.
	\]
	
	Consequently, for
	\[\mathfrak B \in \{\mathrm{SDS}, \mathrm{wSD}, \mathrm{gSDS}, \mathrm{gDS}\} \cup \{ \mathrm{wSD}^\alpha : \alpha \in (0,1] \} \cup \{ \mathrm{SDS}^\alpha : \alpha \in (0,1] \},\]
	\(F \in \mathfrak B_{\mathcal P_{\mathrm{fin}}(X)} \implies F \in \mathfrak B_{\mathcal F[X]}\).
\end{proposition}
\begin{proof}
	Suppose that \(\phi \in \mathbb S\) is such that \(F \circ \phi\) is convergent in \(\mathcal F[X]\).
	Let \(E \in \mathcal F[X]\) be such that \(F \circ \phi \to E\).
	Since the identity mapping \(\mathcal F[X] \to \mathcal P_{\mathrm{fin}}(X)\) is continuous,
	we see that \(F \circ \phi \to E\) in \(\mathcal P_{\mathrm{fin}}(X)\), as well.
	
	Note that we have proved that
	\[
		\{ \phi \in \mathbb S : F \circ \phi \not\in \mathrm{CS}_{\mathcal P_{\mathrm{fin}}(X)} \}
		\subseteq \{ \phi \in \mathbb S : F \circ \phi \not\in \mathrm{CS}_{\mathcal F[X]} \},
	\]
	so we obtain that \(F \in \mathfrak B_{\mathcal P_{\mathrm{fin}}(X)} \implies F \in \mathfrak B_{\mathcal F[X]}\)
	for any \[\mathfrak B \in \{\mathrm{SDS}, \mathrm{gDS}\} \cup \{ \mathrm{SDS}^\alpha : \alpha \in (0,1] \}.\]
	
	Note also that, for \(F \in ([X]^{<\aleph_0})^{\mathbb N}\),
	\[
		\{ \phi \in \mathbb S : F \circ \phi \in \mathrm{SDS}_{\mathcal P_{\mathrm{fin}}}(X) \}
		\subseteq \{ \phi \in \mathbb S : F \circ \phi \in \mathrm{SDS}_{\mathcal F[X]} \}.
	\]
	This establishes that \(F \in \mathfrak B_{\mathcal P_{\mathrm{fin}}(X)} \implies F \in \mathfrak B_{\mathcal F[X]}\)
	for any \[\mathfrak B \in \{ \mathrm{wSD}, \mathrm{gSDS} \} \cup \{ \mathrm{wSD}^\alpha : \alpha \in (0,1] \},\]
	finishing the proof.
\end{proof}

We isolate one of the key ideas from the proof of Theorem 2.6 of \cite{BellaSpadaro}, with a slight generalizing twist,
while also showing it extends to the Vietoris hyperspace of finite subsets, as well.
\begin{definition}
	Let \(\mathcal S\) be a subsemigroup of \(\mathbb S\) and, for a space \(X\),
	\[\mathrm{SDS}^{\mathcal S}_X
	:= \{ s \in X^{\mathbb N} : (\forall \phi \in \mathcal S)(\forall x \in X)(\exists U \in \mathcal N_{X,x})(\exists \psi \in \mathcal S)(\forall n \in \mathbb N)\ 
	s \circ \phi \circ \psi(n) \not\in U \}.\]
\end{definition}
Note that \(\mathrm{SDS}^{\mathbb S}_X = \mathrm{SDS}_X\).
\begin{lemma} \label{lem:PushingSDS}
	Let \(\mathcal S\) be a subsemigroup of \(\mathbb S\) and \(X\) be a space.
	Consider \(\langle x_n : n \in \mathbb N \rangle \in X^{\mathbb N}\) and
	\(\langle F_n : n \in \mathbb N \rangle \in \mathcal F[X]^{\mathbb N}\).
	For each \(n \in \mathbb N\), let \(G_n = F_n \cup \{x_n\}\).
	If \(\langle x_n : n \in \mathbb N \rangle \in \mathrm{SDS}_X^{\mathcal S}\), then
	\(\langle G_n : n \in \mathbb N \rangle \in \mathrm{SDS}^{\mathcal S}_{\mathcal P_{\mathrm{fin}}(X)}\)
	and \(\langle G_n : n \in \mathbb N \rangle \in \mathrm{SDS}^{\mathcal S}_{\mathcal F[X]}\).
\end{lemma}
\begin{proof}
	Let \(s := \langle x_n : n \in \mathbb N \rangle\), \(\langle F_n : n \in \mathbb N \rangle\),
	and \(\langle G_n : n \in \mathbb N \rangle\) be as in the hypotheses.
	We show that \(\langle G_n : n \in \mathbb N \rangle \in \mathrm{SDS}^{\mathcal S}_{\mathcal P_{\mathrm{fin}}(X)}\) and
	\(\langle G_n : n \in \mathbb N \rangle \in \mathrm{SDS}^{\mathcal S}_{\mathcal F[X]}\).
	So let \(\phi \in \mathcal S\) and \(E \in [X]^{<\aleph_0}\) be arbitrary.
	Then let \(\{ y_1, \ldots, y_m \}\) be a faithful enumeration of \(E\).
	Now, since \(\phi \in \mathcal S\), we can let \(U_1\) be a neighborhood of \(y_1\) and \(\psi_1 \in \mathcal S\) be such that
	\(s \circ \phi \circ \psi_1(n) \not\in U_1\) for every \(n \in \mathbb N\).
	
	Suppose that, for some \(1 \leq k < m\), we have defined \(\{ U_j : 1 \leq j \leq k \}\) and \(\{ \psi_j : 1 \leq j \leq k \} \subseteq \mathcal S\).
	Since \(\mathcal S\) is a subsemigroup of \(\mathbb S\), \[\phi_k := \phi \circ \psi_1 \circ \cdots \circ \psi_k \in \mathcal S.\]
	Hence, we can find a neighborhood \(U_{k+1}\) of \(y_{k+1}\) and some \(\psi_{k+1} \in \mathcal S\) such that
	\[s \circ \phi_k \circ \psi_{k+1}(n) \not\in U_{k+1}\]
	for every \(n \in \mathbb N\).
	
	This completes the construction of \(\{ U_j : 1 \leq j \leq m \}\) and
	\(\psi := \psi_1 \circ \cdots \circ \psi_m \in \mathcal S\).
	
	Note that \([U_1, \ldots, U_m]\) is a neighborhood of \(E\) and that, for each \(n \in \mathbb N\),
	\(G_{\phi \circ \psi(n)} \not\in [U_1, \ldots, U_m]\).
	Since \(\phi\) and \(E\) were arbitrary, we see that
	\[\langle G_n : n \in \mathbb N \rangle \in \mathrm{SDS}_{\mathcal P_{\mathrm{fin}}(X)}^{\mathcal S}.\]
	
	In the case for \(\mathcal F[X]\), let \(U = \bigcup_{j=1}^m U_j\) and note that
	\(G_{\phi \circ \psi(n)} \not\in [E, U]\) for any \(n \in \mathbb N\).
	Hence, \[\langle G_n : n \in \mathbb N \rangle \in \mathrm{SDS}_{\mathcal F[X]}^{\mathcal S},\]
	which completes the proof.
\end{proof}
We now note a particular property of \(\mathcal C\), the subsemigroup of \(\mathbb S\) consisting of maps with cofinite range,
in relation to this modified definition of \(\mathrm{SDS}\).
For this, we will say that a subset \(A\) of a topological space \(X\) is \emph{locally finite} if, for every \(x \in X\),
there exists \(U \in \mathcal N_{X,x}\) such that \(\#\{ y \in A : y \in U \} < \aleph_0\).
\begin{proposition} \label{prop:LocallyFinite}
	For any space \(X\), \(s \in \mathrm{SDS}_X^{\mathcal C}\) if and only if \(s\) is finite-to-one and \(\{ s_n : n \in \mathbb N \}\) is locally finite.
\end{proposition}
\begin{proof}
	First, suppose \(s \in \mathrm{SDS}_X^{\mathcal C}\) and let \(x \in X\) be arbitrary.
	By the defining feature of \(\mathrm{SDS}_X^{\mathcal C}\), we can find \(U \in \mathcal N_{X,x}\)
	and some \(\psi \in \mathcal C\) such that \(s \circ \psi(n) \not\in U\) for any \(n \in \mathbb N\).
	Note then that
	\[\{ n \in \mathbb N : s_n \in U \} \subseteq C_\phi,\]
	establishing that \(\{s_n:n\in\mathbb N\}\) is locally finite.
	
	For the fact that \(s\) is finite-to-one, consider \(t \in X^{\mathbb N}\) such that there is some \(\varsigma \in \mathbb S\)
	where \(t \circ \varsigma\) is constant.
	Consider \(x := t \circ \varsigma(1)\) and let \(U \in \mathcal N_{X,x}\) and \(\psi \in \mathcal C\) be arbitrary.
	Note that, since \(\psi\) has cofinite range and \(\varsigma\) has infinite range, there are \(n,m \in \mathbb N\)
	such that \(\psi(n) = \varsigma(m)\).
	Then \(t \circ \psi(n) = t \circ \varsigma(m) = x \in U\).
	Hence, \(t \not\in \mathrm{SDS}_X^{\mathcal C}\).
	
	Now suppose we have \(s \in X^{\mathbb N}\) with the properties that \(s\) is finite-to-one and \(A := \{ s_n : n \in \mathbb N \}\) is locally finite.
	Let \(\phi \in \mathcal C\) and \(x \in X\).
	Then we can let \(U \in \mathcal N_{X,x}\) be such that \(\#\{ y \in A : y \in U \} < \aleph_0\).
	Since \(s\) is finite-to-one, \(\#\{ n \in \mathbb N : s_n \in U \} < \aleph_0\).
	So let
	\[M = \max \{ n \in \mathbb N : s_n \in U \}.\]
	Then we can define \(\psi : \mathbb N \to \mathbb N\) by \(\psi(n) = M+n\).
	Note that \(\psi \in \mathbb S\) and that, for every \(n \in \mathbb N\),
	\[M < \psi(n) \leq \phi(\psi(n)) \implies s \circ \phi \circ \psi(n) \not\in U.\]
	Since \(\psi \in \mathcal C\), we see that \(s \in \mathrm{SDS}_X^{\mathcal C}\).
\end{proof}
\begin{corollary}
	For any space \(X\),
	\[
		\vec{\mathsf{G}}_1(\mathscr T_X , \mathrm{CD}_X^\ast)
		\leq_{\mathrm{II}}^+ \vec{\mathsf{G}}_1(\mathscr T_X , \mathrm{SDS}_X^{\mathcal C})
		\leq_{\mathrm{II}}^+ \vec{\mathsf{G}}_1(\mathscr T_X , \mathrm{SDS}_X).
	\]
	If \(X\) is assumed to be \(T_1\), then
	\[\vec{\mathsf{G}}_1(\mathscr T_X , \mathrm{SDS}_X^{\mathcal C}) \leq_{\mathrm{II}}^+ \mathsf{G}_1(\mathscr T_X , \mathrm{CD}_X).\]
\end{corollary}
\begin{proof}
	For the first inequality, it suffices to note that \(\mathrm{CD}_X^\ast \subseteq \mathrm{SDS}_X^{\mathcal C} \subseteq \mathrm{SDS}_X\).
	For the second inequality, note that, in a \(T_1\) space, any countable set of points which is locally finite is actually closed and relatively discrete.
\end{proof}

We show that the technique in the proof of \cite[Theorem 2.6]{BellaSpadaro} extends well beyond the original context to cover strategic transferral,
the WSHD property, and the \(\mathcal P_{\mathrm{fin}}(X)\) context.
\begin{theorem} \label{thm:HyperSDS}
	Let \(\mathcal S\) be a subsemigroup of \(\mathbb S\) and \(\mathcal E \subseteq \mathbb S\).
	For a space \(Y\), let
	\[
		\mathfrak T(\mathcal E, \mathcal S)_Y = \left\{ s \in Y^{\mathbb N} : \exists \phi \in \mathcal E\ \left(s \circ \phi \in \mathrm{SDS}_Y^{\mathcal S}\right) \right\}.
	\]
	Then, for \(\mathbf H(X) \in \{ \mathcal P_{\mathrm{fin}}(X), \mathcal F[X] \}\),
	\[
		\vec{\mathsf G}_1(\mathscr T_X, \mathfrak T(\mathcal E, \mathcal S)_X) \leq_{\mathrm{II}}^+
		\vec{\mathsf G}_1(\mathscr T_{\mathbf H(X)} , \mathfrak T(\mathcal E, \mathcal S)_{\mathbf H(X)}).
	\]
	
	Consequently, for any \(\mathfrak B \in \{ \mathrm{SDS}, \mathrm{wSD}, \mathrm{gSDS} \} \cup \{ \mathrm{wSD}^\alpha : \alpha \in (0,1] \}\),
	\[
		\vec{\mathsf G}_1(\mathscr T_X, \mathfrak B_X) \leq_{\mathrm{II}}^+
		\vec{\mathsf G}_1(\mathscr T_{\mathbf H(X)} , \mathfrak B_{\mathbf H(X)}).
	\]
\end{theorem}
\begin{proof}
	We proceed by cases.
	
	We start with the case for \(\mathcal P_{\mathrm{fin}}(X)\).
	Fix, for each open subset \(W\) of \(\mathcal P_{\mathrm{fin}}(X)\), a basic neighborhood contained in \(W\).
	Let this choice be realized by open subsets \(\{ U_{W,1}, \ldots, U_{W,m_W} \}\) of \(X\).
	For each \(W \in \mathscr T_{\mathcal P_{\mathrm{fin}}(X)}\), let \(F_W \in \mathcal P_{\mathrm{fin}}(X)\) be such that
	\[F_W \in [U_{W,1}, \ldots, U_{W,m_W}] \subseteq W.\]
	
	Now, define \(\overleftarrow{\mathsf T}_{\mathrm{I}} : \mathscr T_{\mathcal P_{\mathrm{fin}}(X)} \to \mathscr T_X\) by
	\[\overleftarrow{\mathsf T}_{\mathrm{I}}(W) = \bigcup_{j=1}^{n_W} U_{W,j}.\]
	Then define \(\overrightarrow{\mathsf T}_{\mathrm{II}} : X \times \mathscr T_{\mathcal P_{\mathrm{fin}}(X)} \to \mathcal P_{\mathrm{fin}}(X)\) by
	\[\overrightarrow{\mathsf T}_{\mathrm{II}}(x, W) = F_W \cup \{x\}.\]
	
	Now consider an arbitrary sequence \(\langle W_n : n \in \mathbb N \rangle\) of \(\mathscr T_{\mathcal P_{\mathrm{fin}}(X)}\).
	Then suppose we have \(x_n \in \overleftarrow{\mathsf T}_{\mathrm{I}}(W_n)\) for each \(n \in \mathbb N\)
	such that \(s := \langle x_n : n \in \mathbb N \rangle \in \mathfrak T(\mathcal E, \mathcal S)_X\).
	Then we can let \(\phi \in \mathcal E\) be such that \(s \circ \phi \in \mathrm{SDS}_X^{\mathcal S}\).
	By Lemma \ref{lem:PushingSDS},
	\[\left\langle \overrightarrow{\mathsf T}_{\mathrm{II}}(x_{\phi(n)}, W_{\phi(n)}) : n \in \mathbb N \right\rangle
	\in \mathrm{SDS}^{\mathcal S}_{\mathcal P_{\mathrm{fin}}(X)}.\]
	Thus,
	\[\left\langle \overrightarrow{\mathsf T}_{\mathrm{II}}(x_{n}, W_{n}) : n \in \mathbb N \right\rangle
	\in \mathfrak T(\mathcal E, \mathcal S)_{\mathcal P_{\mathrm{fin}}(X)}.\]
	Then, applying Theorem \ref{thm:TranslationTheorem} establishes that
	\[
		\vec{\mathsf G}_1(\mathscr T_X, \mathfrak T(\mathcal E, \mathcal S)_X) \leq_{\mathrm{II}}^+
		\vec{\mathsf G}_1(\mathscr T_{\mathcal P_{\mathrm{fin}}(X)} , \mathfrak T(\mathcal E, \mathcal S)_{\mathcal P_{\mathrm{fin}}(X)}).
	\]

	Now we attend to the case for \(\mathcal F[X]\).
	For every \(W \in \mathscr T_{\mathcal F[X]}\), fix \(F_W \in \mathcal F[X]\) and \(U_W \in \mathscr T_X\)
	such that \([F_W,U_W] \subseteq W\).
	Now define \(\overleftarrow{\mathsf T}_{\mathrm{I}} : \mathscr T_{\mathcal F[X]} \to \mathscr T_X\) by
	\[\overleftarrow{\mathsf T}_{\mathrm{I}}(W) = U_W.\]
	Then define \(\overrightarrow{\mathsf T}_{\mathrm{II}} : X \times \mathscr T_{\mathcal F[X]} \to \mathcal F[X]\) by
	\[\overrightarrow{\mathsf T}_{\mathrm{II}}(x, W) = F_W \cup \{x\}.\]
	
	Consider an arbitrary sequence \(\langle W_n : n \in \mathbb N \rangle\) of \(\mathscr T_{\mathcal F[X]}\).
	Suppose we have \(x_n \in \overleftarrow{\mathsf T}_{\mathrm{I}}(W_n)\) for each \(n \in \mathbb N\)
	such that \(s := \langle x_n : n \in \mathbb N \rangle \in \mathfrak T(\mathcal E, \mathcal S)_X\).
	Then we can let \(\phi \in \mathcal E\) be such that \(s \circ \phi \in \mathrm{SDS}_X^{\mathcal S}\).
	By Lemma \ref{lem:PushingSDS},
	\[\left\langle \overrightarrow{\mathsf T}_{\mathrm{II}}(x_{\phi(n)}, W_{\phi(n)}) : n \in \mathbb N \right\rangle
	\in \mathrm{SDS}^{\mathcal S}_{\mathcal F[X]}.\]
	Thus,
	\[\left\langle \overrightarrow{\mathsf T}_{\mathrm{II}}(x_{n}, W_{n}) : n \in \mathbb N \right\rangle
	\in \mathfrak T(\mathcal E, \mathcal S)_{\mathcal F[X]}.\]
	Then, applying Theorem \ref{thm:TranslationTheorem} establishes that
	\[
		\vec{\mathsf G}_1(\mathscr T_X, \mathfrak T(\mathcal E, \mathcal S)_X) \leq_{\mathrm{II}}^+
		\vec{\mathsf G}_1(\mathscr T_{\mathcal F[X]} , \mathfrak T(\mathcal E, \mathcal S)_{\mathcal F[X]}).
	\]
		
	For the remaining claim in the statement of the theorem, note that,
	\begin{itemize}
		\item
		when \(\mathcal E = \{ \mathrm{id} \}\), where \(\mathrm{id}\) is the identity mapping \(\mathbb N \to \mathbb N\),
		and \(\mathcal S = \mathbb S\), \(\mathfrak T(\mathcal E, \mathcal S) = \mathrm{SDS}\);
		\item
		when \(\mathcal E = \mathbb S\) and \(\mathcal S = \mathbb S\), \(\mathfrak T(\mathcal E, \mathcal S) = \mathrm{wSD}\); and
		\item
		when \(\mathcal E = \mathbf D_\alpha^\ast\) and \(\mathcal S = \mathbb S\), \(\mathfrak T(\mathcal E, \mathcal S) = \mathrm{wSD}^\alpha\).
	\end{itemize}
	The case for \(\mathrm{gSDS}\) follows using the same arguments above, and just noting that
	Lemma \ref{lem:PushingSDS} guarantees that a comeager subset of \(\phi \in \mathbb S\) for which
	\(\langle x_{\phi(n)} : n \in \mathbb N \rangle \in \mathrm{SDS}_X\) will also witness that the translated
	selections are in the corresponding \(\mathrm{SDS}_{\mathbf H(X)}\).
\end{proof}
\begin{corollary}
	If a space \(X\) is SHD (resp., WSHD), then \(\mathcal P_{\mathrm{fin}}(X)\) is SHD (resp., WSHD).
\end{corollary}
Note that the inequality in Theorem \ref{thm:HyperSDS} doesn't generally reverse in the context of \(\mathcal F[X]\).
\begin{example}
	For any uncountable cardinal \(\kappa\), \(\mathcal F[\Sigma(2,0,\kappa)]\) is SHD (\cite[Theorem 5]{RemarksOnSHD}) but \(\Sigma(2,0,\kappa)\) is not SHD
	since it is sequentially compact (as discussed in Example \ref{example:SigmaProduct}).
	So it is not generally the case that \(\vec{\mathsf G}_1(\mathscr T_{\mathcal F[X]} , \mathrm{SDS}_{\mathcal F[X]})
	\leq_{\mathrm{II}} \vec{\mathsf G}_1(\mathscr T_{X} , \mathrm{SDS}_{X})\).
\end{example}
However, we have not identified any similar example for the \(\mathcal P_{\mathrm{fin}}(X)\) context.
\begin{question} \label{question:PfinEquivalence}
	Is it possible that \(\vec{\mathsf G}_1(\mathscr T_{\mathcal P_{\mathrm{fin}}(X)} , \mathrm{SDS}_{\mathcal P_{\mathrm{fin}}(X)})
	\leq_{\mathrm{II}} \vec{\mathsf G}_1(\mathscr T_{X} , \mathrm{SDS}_{X})\)?
	If so, to what level of generality can we take the \(\mathrm{SDS}\) operator and maintain the inequality?
\end{question}
One obstacle in an attempted proof of such an inequality is that
simply assuming that \(\langle F_n : n \in \mathbb N \rangle \in \mathrm{SDS}_{\mathcal P_{\mathrm{fin}}(X)}\) does not guarantee a way to select
\(x_n \in F_n\) for each \(n \in \mathbb N\) so that \(\langle x_n : n \in \mathbb N \rangle \in \mathrm{SDS}_X\).
Indeed, consider \(X = [0,1]\) and, for each \(n \in \mathbb N\),
let \(F_n = \left\{ \frac{j}{2^n} : 0 \leq j \leq 2^n \right\}\) and observe that
\(\langle F_n : n \in \mathbb N \rangle \in \mathrm{SDS}_{\mathcal P_{\mathrm{fin}}(X)}\).
However, since \(X\) is sequentially compact, any selection \(x_n \in F_n\) for each \(n \in \mathbb N\)
has the property that \(\langle x_n : n \in \mathbb N \rangle\) has a convergent subsequence.

Of course, an affirmative answer to the following would provide a negative answer to Question \ref{question:PfinEquivalence}.
\begin{question}
	Is there some uncountable cardinal \(\kappa\) for which \(\mathcal P_{\mathrm{fin}}(\Sigma(2,0,\kappa))\) is SHD?
\end{question}

What about a more direct relationship between the hyperspaces?
\begin{question}
	In the context of Theorem \ref{thm:HyperSDS},
	is it true that
	\[\vec{\mathsf G}_1(\mathscr T_{\mathcal P_{\mathrm{fin}}(X)}, \mathfrak T(\mathcal E, \mathcal S)_{\mathcal P_{\mathrm{fin}}(X)}) \leq_{\mathrm{II}}^+
	\vec{\mathsf G}_1(\mathscr T_{\mathcal F[X]} , \mathfrak T(\mathcal E, \mathcal S)_{\mathcal F[X]})?\]
\end{question}
One approach for establishing a strategic translation is as follows.
For a basic open subset \([F_n,U_n]\) of \(\mathcal F[X]\), we can choose \(U_{n,x} \in \mathcal N_{X,x}\), \(U_{n,x} \subseteq U_n\), for each \(x \in F_n\) and form
\([U_{n,x} : x \in F]\), a basic open subset of \(\mathcal P_{\mathrm{fin}}(X)\).
Given some \(H_n \in [U_{n,x} : x \in F]\), we can then push forward to \(G_n := F_n \cup H_n\).
Note that \(G_n \in [F_n,U_n]\).
However, just under the assumption that \(\langle H_n : n \in \mathbb N \rangle\) has some divergence property, it is not clear
that we can guarantee that \(\langle G_n : n \in \mathbb N \rangle\) must also have that divergence property.

\subsection{Consequences for Rings of Continuous Functions}

\begin{definition}
	For a space \(X\), we say that a family \(\mathcal I\) of nonempty closed subsets of \(X\) is a \emph{proper ideal of closed sets} if
	\begin{itemize}
		\item
		every \(A \in \mathcal I\) is a proper subset of \(X\),
		\item
		for \(A , B \in \mathcal I\), \(A \cup B \in\mathcal I\), and
		\item
		\(\{ \{x\} : x \in X \} \subseteq \mathcal I\).
	\end{itemize}
\end{definition}
\begin{definition}
	Let \(X\) be a space and \(\mathcal I\) be a proper ideal of closed subsets of \(X\).
	We say that \(X\) is \emph{functionally \(\mathcal I\)-normal} if, for \(A \in \mathcal I\)
	and \(U \subsetneq X\) open with \(A \subseteq U\), there exists a continuous function \(f : X \to \mathbb R\)
	so that \(f[A] = \{0\}\) and \(f[X \setminus U] = \{1\}\).
\end{definition}
Note that, if \(X\) is Tychonoff and \(\mathcal I\) is an ideal that consists of compact sets, then \(X\) is functionally \(\mathcal I\)-normal.

Topologies relative to uniform convergence on a select class of subsets are well-known and studied in depth.
For a general development of uniform convergence topologies and relevant commentary, see \cite{Osipov,OsipovCoincidence};
for recent developments concerning uniform convergence on not necessarily closed subsets, see \cite{UniformLindelof,UniformSeparable}.
\begin{definition}
	For a space \(X\) and a proper ideal of closed sets \(\mathcal I\), we let \(C_{\mathcal I}(X)\) denote the set \(C(X)\) of continuous
	real-valued functions on \(X\) with the topology generated by sets of the form
	\[[f; A, \varepsilon] = \{ g \in C(X) : \sup\{ |f(x) - g(x)| : x \in A\} < \varepsilon \},\]
	where \(f \in C(X)\), \(A \in \mathcal I\), and \(\varepsilon > 0\).
	That is, \(C_{\mathcal I}(X)\) has as its topology the topology of uniform convergence on \(\mathcal I\).
\end{definition}
\begin{remark}
	Note that \(C_{\mathcal I}(X)\) with pointwise addition forms a topological group.
\end{remark}
When \(\mathcal I\) is the ideal of
\begin{itemize}
	\item
	finite subsets, we use the notation \(C_p(X)\) instead of \(C_{\mathcal I}(X)\).
	\item
	compact subsets, we use the notation \(C_k(X)\) instead of \(C_{\mathcal I}(X)\).
\end{itemize}

Though some of the implications for Theorem \ref{thm:MainFunctionTheorem} are proved elsewhere in varying levels of generality
(see, for example, \cite{CHContinuousFunctions} and relevant commentary and references therein), we include proofs which avoid some notational overhead for the convenience of the reader.

Recall that a subset \(A\) of a space \(X\) is said to be \emph{\(\mathbb R\)-bounded} if, for every \(f \in C(X)\),
\(f[A]\) is bounded.
\begin{theorem} \label{thm:MainFunctionTheorem}
	Suppose \(X\) is a space and \(\mathcal I\) is a proper ideal of closed subsets of \(X\).
	If \(X\) is functionally \(\mathcal I\)-normal and \(\mathcal I\) consists of \(\mathbb R\)-bounded sets, then the following are equivalent:
	\begin{enumerate}[label=(\roman*),ref=(\roman*)]
		\item \label{SpaceNotA}
		For every \(\{ A_n : n \in \mathbb N \} \subseteq \mathcal I\), there exists \(E \in \mathcal I\) such that,
		for every \(n \in \mathbb N\), \(E \not\subseteq A_n\).
		\item \label{SpecialCoveringProperty}
		For every \(\{ A_n : n \in \mathbb N \} \subseteq \mathcal I\), there exist \(\{ U_n : n \in \mathbb N \} \subseteq \mathscr T_X\) and
		\(E \in \mathcal I\) such that \(E \not\subseteq U_n\) for any \(n \in \mathbb N\).
		\item \label{CNotFirstCountable}
		\(C_{\mathcal I}(X)\) is not first-countable.
		\item \label{CNotMetrizable}
		\(C_{\mathcal I}(X)\) is not metrizable.
		\item \label{CAIDS}
		\(C_{\mathcal I}(X)\) is injectively discretely selective.
		\item \label{CADS}
		\(C_{\mathcal I}(X)\) is discretely selective.
		\item \label{CASHD}
		\(C_{\mathcal I}(X)\) is SHD.
		\item \label{CAWSHD}
		\(C_{\mathcal I}(X)\) is WSHD.
	\end{enumerate}
\end{theorem}
\begin{proof}
	Note that \ref{CAIDS}\(\implies\)\ref{CADS}, \ref{CAIDS}\(\implies\)\ref{CASHD}, and \ref{CASHD}\(\implies\)\ref{CAWSHD} are immediate.
	The equivalence of \ref{CNotFirstCountable} and \ref{CNotMetrizable} is a consequence of the well-known Birkhoff-Kakutani Theorem.
	
	\ref{CADS}\(\implies\)\ref{CNotFirstCountable}:
	Suppose \(C_{\mathcal I}(X)\) is first-countable, and note that \(\mathbf 0\) is not isolated.
	So, by Proposition \ref{prop:CountableLocalPiBase}, \(C_{\mathcal I}(X)\) is not discretely selective.
	
	\ref{CNotFirstCountable}\(\implies\)\ref{SpaceNotA}:
	By way of contrapositive, suppose there is \(\{A_n:n\in\mathbb N\} \subseteq \mathcal I\) such that, for each \(E \in \mathcal I\),
	there exists \(n \in \mathbb N\) with \(E \subseteq A_n\).
	Then note that \(\{ [\mathbf 0; A_n, 2^{-m}] : n,m \in \mathbb N \}\) forms a countable basis at \(\mathbf 0\), the constant zero function.
	That is, \(C_{\mathcal I}(X)\) is first-countable.
	
	\ref{SpaceNotA}\(\implies\)\ref{SpecialCoveringProperty}:
	Let \(\{ A_n : n \in \mathbb N \} \subseteq \mathcal I\) be arbitrary.
	Then let \(E \in \mathcal I\) be such that \(E \not\subseteq A_n\) for any \(n \in \mathbb N\).
	Choose \(x_n \in E \setminus A_n\) and let \(U_n = X \setminus \{x_n\}\) for each \(n\).
	Note that \(A_n \subseteq U_n\) and that \(E \not\subseteq U_n\) for any \(n \in \mathbb N\).
	
	\ref{SpecialCoveringProperty}\(\implies\)\ref{CAIDS}:
	Define \[\Phi : C(X) \times \mathcal I \times \mathscr T_X \times \mathbb N \to C(X)\] to be a choice such that,
	whenever \(A \subseteq U \subsetneq X\),
	\[\Phi(f,A,U,n)\restriction_{A} = f \restriction_A \text{ and } \Phi(f,A,U,n)[X \setminus U] = \{n\};\]
	whenever \(A \not\subseteq U\) or \(U = X\), \(\Phi(f,A,U,n) = f\).
	Now consider an arbitrary sequence \(\langle [f_n ; A_n, \varepsilon_n ] : n \in \mathbb N \rangle\)
	of basic neighborhoods of \(C_{\mathcal I}(X)\).
	By \ref{SpecialCoveringProperty}, let \(\{ U_n : n \in \mathbb N \} \subseteq \mathscr T_X\) and \(E \in \mathcal I\)
	have the guaranteed property.
	By construction, \(n \mapsto g_n := \Phi(f_n,A_n,U_n,\varepsilon_n)\), \(\mathbb N \to C_{\mathcal I}(X)\), is injective.
	So we need only verify that \(\{ g_n : n \in \mathbb N \}\) is closed and relatively discrete in \(C_{\mathcal I}(X)\).
	To do this, we show that \(\{ g_n : n \in \mathbb N \}\) is locally finite.
	
	So let \(h \in C(X)\) be arbitrary and let \(M \in \mathbb N\) be such that \(\sup\{ h(x) : x \in E \} + 1 < M\).
	For any \(n \geq M\), we can let \(x_0 \in E \setminus U_n\).
	Then note that \(h(x_0) + 1 < M \leq n = g_n(x_0)\).
	Hence, \(1 < \sup \{ |h(x) - g_n(x)| : x \in E \}\).
	Thus, \(\#\{n \in \mathbb N : g_n \in [ h; E, 1 ] \} < \aleph_0\).
	
	\ref{CAWSHD}\(\implies\)\ref{CNotFirstCountable}:
	Note that, since \(C_{\mathcal I}(X)\) is assumed to be WSHD, \(\mathcal F[C_{\mathcal I}(X)]\) is WSHD by Theorem \ref{thm:HyperSDS}.
	Then \(C_{\mathcal I}(X)\) is not first-countable by \cite[Theorem 12]{RemarksOnSHD}.
\end{proof}
\begin{corollary} \label{cor:CpTheorem}
	For any Tychonoff space \(X\), the following are equivalent:
	\begin{enumerate}[label=(\roman*),ref=(\roman*)]
		\item
		\(X\) is uncountable.
		\item
		\(C_p(X)\) is not metrizable.
		\item
		\(C_p(X)\) is injectively discretely selective.
		\item
		\(C_p(X)\) is discretely selective.
		\item
		\(C_p(X)\) is SHD.
		\item
		\(C_p(X)\) is WSHD.
	\end{enumerate}
\end{corollary}
\begin{corollary} \label{cor:CkTheorem}
	For any Tychonoff space \(X\), the following are equivalent:
	\begin{enumerate}[label=(\roman*),ref=(\roman*)]
		\item
		\(X\) is not hemicompact.
		\item
		\(C_k(X)\) is not metrizable.
		\item
		\(C_k(X)\) is injectively discretely selective.
		\item
		\(C_k(X)\) is discretely selective.
		\item
		\(C_k(X)\) is SHD.
		\item
		\(C_k(X)\) is WSHD.
	\end{enumerate}
\end{corollary}
We elaborate now on the general non-transfer of the discretely selective and SHD properties between \(X\) and \(C_{\mathcal I}(X)\).
\begin{example}
	Note that \(\mathbb R\) is neither discretely selective nor SHD.
	However, by Corollary \ref{cor:CpTheorem}, \(C_p(X)\) is both discretely selective and SHD.
\end{example}
\begin{example}
	Note that \(\mathbb N\) is discretely selective but \(C_p(\mathbb N)\) is not discretely selective.
\end{example}
\begin{question}
	Is there a discretely selective space \(X\) without isolated points and an ideal \(\mathcal I\) consisting of \(\mathbb R\)-bounded
	sets for which \(C_{\mathcal I}(X)\) is not discretely selective?
\end{question}
\begin{example}
	If \(X\) is a compact SHD space (see, for example, Examples \ref{ex:CechStoneOmega} and \ref{ex:CirclePower}),
	then \(C_k(X)\) is clearly first-countable, and hence neither discretely selective nor SHD.
	Note however that \(X\) is not discretely selective.
\end{example}

\section*{Acknowledgements}

The author would like the thank the anonymous referee for comments leading to the paper's improvement.

\providecommand{\bysame}{\leavevmode\hbox to3em{\hrulefill}\thinspace}
\providecommand{\MR}{\relax\ifhmode\unskip\space\fi MR }
\providecommand{\MRhref}[2]{%
  \href{http://www.ams.org/mathscinet-getitem?mr=#1}{#2}
}
\providecommand{\href}[2]{#2}

\end{document}